\documentclass[a4paper]{amsart}

\usepackage{setspace}
\usepackage[totalwidth=14cm,totalheight=22cm]{geometry}
\usepackage[T1]{fontenc}
\usepackage[dvips]{graphicx}
\usepackage{epsfig}
\usepackage{amsmath,amssymb,amscd,amsthm,nicefrac}
\usepackage{subfigure}
\usepackage[latin1]{inputenc}
\usepackage{latexsym}
\usepackage{graphics}
\usepackage{colortbl}
\usepackage{color}
\usepackage{ae}
\usepackage{enumitem}
\usepackage[all]{xy}
\usepackage[bf,small]{caption}
\usepackage{floatflt}
\usepackage{float}

\usepackage{chngcntr}
\counterwithin{figure}{section}

\DeclareMathOperator\arctanh{arctanh}

\makeatletter
\newtheorem*{rep@theorem}{\rep@title}
\newcommand{\newreptheorem}[2]{%
\newenvironment{rep#1}[1]{%
 \def\rep@title{#2 \ref{##1}}%
 \begin{rep@theorem}}%
 {\end{rep@theorem}}}
\makeatother

\makeatletter
\newtheorem*{rep@cor}{\rep@title}
\newcommand{\newrepcor}[2]{%
\newenvironment{rep#1}[1]{%
 \def\rep@title{#2 \ref{##1}}%
 \begin{rep@cor}}%
 {\end{rep@cor}}}
\makeatother

\makeatletter
\newtheorem*{rep@prop}{\rep@title}
\newcommand{\newrepprop}[2]{%
\newenvironment{rep#1}[1]{%
 \def\rep@title{#2 \ref{##1}}%
 \begin{rep@prop}}%
 {\end{rep@prop}}}
\makeatother

\newtheorem{cor}{Corollary}[section]
\newtheorem*{cor*}{Corollary}
\newtheorem{corx}{Corollary}

\newtheorem{theorem}[cor]{Theorem}
\newtheorem{thmx}{Theorem}

\newtheorem{prop}[cor]{Proposition}

\newrepcor{cor}{Corollary}
\newreptheorem{theorem}{Theorem}
\newrepprop{prop}{Proposition}

\newtheorem{lemma}[cor]{Lemma}

\theoremstyle{definition}
\newtheorem{defi}[cor]{Definition}
\theoremstyle{remark}
\newtheorem{remark}[cor]{Remark}

\newcommand*{\bullets}{\raisebox{-0.25ex}{\scalebox{1.5}{$\cdot$}}}

\newcommand{\C}{{\mathbb C}}

\newcommand{\R}{{\mathbb R}}

\newcommand{\Hyp}{\mathbb{H}}
\newcommand{\dS}{\mathrm{d}\mathbb{S}}

\newcommand{\grad}{\operatorname{grad}}
\newcommand{\isom}{\mathrm{Isom}}

\newcommand{\II}{I\hspace{-0.1cm}I}
\newcommand{\III}{I\hspace{-0.1cm}I\hspace{-0.1cm}I}

\newcommand{\rar}{\rightarrow}

\newcommand{\SO}{\mathrm{SO}}

\newcommand{\D}{\mathbb{D}}

\def\Hess{\mathrm{Hess}}

\def\Teich{\mathcal{T}}

\begin{document}

\setcounter{secnumdepth}{3}
\setcounter{tocdepth}{2}

\title{Minimal discs in hyperbolic space bounded by a quasicircle at infinity}

\author{Andrea Seppi} 
\address{A. Seppi: Dipartimento di Matematica ``Felice Casorati", Universit\`{a} degli Studi di Pavia, Via Ferrata 1, 27100, Pavia, Italy.} \email{andrea.seppi01@ateneopv.it}

\maketitle



\begin{abstract}
We prove that the supremum of principal curvatures of a minimal embedded disc in hyperbolic three-space spanning a quasicircle in the boundary at infinity is estimated in a sublinear way by the norm of the quasicircle in the sense of universal Teichm\"uller space, if the quasicircle is sufficiently close to being the boundary of a totally geodesic plane. As a by-product we prove that there is a universal constant C independent of the genus such that if the Teichm\"uller distance between the ends of a quasi-Fuchsian manifold $M$ is at most C, then $M$ is almost-Fuchsian. 
The main ingredients of the proofs are estimates on the convex hull of a minimal surface and Schauder-type estimates to control principal curvatures.
\end{abstract}


\section{Introduction}

Let $\Hyp^3$ be hyperbolic three-space and $\partial_\infty \Hyp^3$ be its boundary at infinity. A surface $S$ in hyperbolic space is minimal if its principal curvatures at every point $x$ have opposite values. We will denote the principal curvatures by $\lambda$ and $-\lambda$, where $\lambda=\lambda(x)$ is a nonnegative function on $S$. It was proved by Anderson (\cite[Theorem 4.1]{anderson}) that for every Jordan curve $\Gamma$ in $\partial_\infty \Hyp^3$ there exists a minimal embedded disc $S$ whose boundary at infinity coincides with $\Gamma$. 
It can be proved that if the supremum $||\lambda||_\infty$ of the principal curvatures of $S$ is in $(-1,1)$, then $\Gamma=\partial_\infty S$ is a quasicircle, namely $\Gamma$ is the image of a round circle under a quasiconformal map of the sphere at infinity.

However, uniqueness does not hold in general. 
Anderson proved the existence of a Jordan curve $\Gamma\subset\partial_\infty \Hyp^3$ invariant under the action of a quasi-Fuchsian group $G$ spanning several distinct minimal embedded discs, see \cite[Theorem 5.3]{anderson}. In this case, $\Gamma$ is a quasicircle and coincides with the limit set of $G$. More recently in \cite{zeno} invariant curves spanning an arbitrarily large number of minimal discs were constructed.
On the other hand, if the supremum of the principal curvatures of a minimal embedded disc $S$ satisfies $||\lambda||_\infty\in(-1,1)$ then, by an application of the maximum principle, $S$ is the unique minimal disc asymptotic to the quasicircle $\Gamma=\partial_\infty S$.

The aim of this paper is to study the supremum $||\lambda||_\infty$ of the principal curvatures of a minimal embedded disc, in relation with the norm of the quasicircle at infinity, in the sense of universal Teichm\"uller space. The relations we obtain are interesting for ``small'' quasicircles, that are close in universal Teichm\"uller space to a round circle. The main result of this paper is the following:

\begin{thmx} \label{hyperbolic close to identity}
There exist universal constants $K_0>1$ and $C>0$ such that every minimal embedded disc in $\Hyp^3$ with boundary at infinity a $K$-quasicircle $\Gamma\subset\partial_\infty\Hyp^3$, with $1\leq K\leq K_0$, has principal curvatures bounded by $$||\lambda||_\infty\leq C\log K\,.$$
\end{thmx}

Recall that the minimal disc with prescribed quasicircle at infinity is unique if $||\lambda||_\infty<1$. Hence we can draw the following consequence, by choosing $K_0'<\min\{K_0,e^{1/C}\}$:

\begin{thmx} \label{theorem uniqueness}
There exists a universal constant $K'_0$ such that every $K$-quasicircle $\Gamma\subset\partial_\infty\Hyp^3$ with $K\leq K'_0$ is the boundary at infinity of a unique minimal embedded disc.
\end{thmx}

\subsection*{Applications to quasi-Fuchsian manifolds}

Theorem \ref{hyperbolic close to identity} has a direct application to quasi-Fuchsian manifolds. Recall that a quasi-Fuchsian manifold $M$ is isometric to the quotient of $\Hyp^3$ by a quasi-Fuchsian group $G$, isomorphic to the fundamental group of a closed surface $\Sigma$, whose limit set is a Jordan curve $\Gamma$ in $\partial_\infty\Hyp^3$. The topology of $M$ is $\Sigma\times\R$. We denote by $\Omega_+$ and $\Omega_-$ the two connected components of $\partial_\infty\Hyp^3\setminus \Gamma$. Then $\Omega_{+}/G$ and $\Omega_{-}/G$ inherit natural structures of Riemann surfaces on $\Sigma$ and therefore determine two points of $\Teich(\Sigma)$, the Teichm\"uller space of $\Sigma$. Let $d_{\Teich(\Sigma)}$ denote the Teichm\"uller distance on $\Teich(\Sigma)$.

\begin{corx} \label{cor teich dist}
There exist universal constants $C>0$ and $d_0>0$ such that, for every quasi-Fuchsian manifold $M=\Hyp^3/G$ with $d_{\Teich(\Sigma)}(\Omega_+/G,\Omega_-/G)<d_0$ and every minimal surface $S$ in $M$ homotopic to $\Sigma\times\{0\}$, the supremum of the principal curvatures of $S$ satisfies:
$$||\lambda||_\infty\leq C d_{\Teich(\Sigma)}(\Omega_+/G,\Omega_-/G)\,.$$
\end{corx}

Indeed, under the hypothesis of Corollary \ref{cor teich dist}, the Teichm\"uller map from one hyperbolic end of $M$ to the other is $K$-quasiconformal for $K\leq e^{2d_0}$. Hence the lift to the universal cover $\Hyp^3$ of any closed minimal surface in $M$ is a minimal embedded disc with boundary at infinity a $K$-quasicircle, namely the limit set of the corresponding quasi-Fuchsian group. Choosing $d_0=(1/2)\log K_0$, where $K_0$ is the constant of Theorem \ref{hyperbolic close to identity}, and choosing $C$ as in Theorem \ref{hyperbolic close to identity} (up to a factor $2$ which arises from the definition of Teichm\"uller distance), the statement of Corollary \ref{cor teich dist} follows.

We remark here that the constant $C$ of Corollary \ref{cor teich dist} is independent of the genus of $\Sigma$.

A quasi-Fuchsian manifold contaning a closed minimal surface with principal curvatures in $(-1,1)$ is called almost-Fuchsian, according to the definition given in \cite{Schlenker-Krasnov}. The minimal surface in an almost-Fuchsian manifold is unique, by the above discussion, as first observed by Uhlenbeck (\cite{uhlenbeck}). Hence, applying Theorem \ref{theorem uniqueness} to the case of quasi-Fuchsian manifolds, the following corollary is proved.

\begin{corx} \label{corollary teichmuller almost fuchsian}
If the Teichm\"uller distance between the conformal metrics at infinity of a quasi-Fuchsian manifold $M$ is smaller than a universal constant $d_0'$, then $M$ is almost-Fuchsian.
\end{corx}

Indeed, it suffices as above to pick $d_0'=(1/2)\log K_0'$, which is again independent on the genus of $\Sigma$. By Bers' Simultaneous Uniformization Theorem, the Riemann surfaces $\Omega_{\pm}/G$ determine the manifold $M$. Hence the space $\mathcal{QF}(\Sigma)$ of quasi-Fuchsian manifolds homeomorphic to $\Sigma\times\R$, considered up to isometry isotopic to the identity, can be identified to $\Teich(\Sigma)\times \Teich(\Sigma)$. Under this identification, the subset of $\mathcal{QF}(\Sigma)$ composed of Fuchsian manifolds, which we denote by $\mathcal{F}(\Sigma)$, coincides with the diagonal in $\Teich(\Sigma)\times \Teich(\Sigma)$. Let us denote by $\mathcal{AF}(\Sigma)$ the subset of $\mathcal{QF}(\Sigma)$ composed of almost-Fuchsian manifolds. Corollary \ref{corollary teichmuller almost fuchsian} can be restated in the following way:

\begin{corx} \label{corollary neighborhood}
There exists a uniform neighborhood $N(\mathcal{F}(\Sigma))$ of the Fuchsian locus $\mathcal{F}(\Sigma)$ in $\mathcal{QF}(\Sigma)\cong \Teich(\Sigma)\times \Teich(\Sigma)$ such that $N(\mathcal{F}(\Sigma))\subset \mathcal{AF}(\Sigma)$.
\end{corx}
 



We remark that Corollary \ref{cor teich dist} is a partial converse of results presented in \cite{ghw}, giving a bound on the Teichm\"uller distance between the hyperbolic ends of an almost-Fuchsian manifold in terms of the maximum of the principal curvatures.
Another invariant which has been studied in relation with the properties of minimal surfaces in hyperbolic space is the Hausdorff dimension of the limit set. 
Corollary \ref{cor teich dist} and Corollary \ref{corollary teichmuller almost fuchsian} can be compared with the following theorem given in \cite{sanders2}: for every $\epsilon$ and $\epsilon_0$ there exists a constant $\delta=\delta(\epsilon,\epsilon_0)$ such that any stable minimal surface with injectivity radius bounded by $\epsilon_0$ in a quasi-Fuchsian manifold $M$ are in $(-\epsilon,\epsilon)$ provided the Hausdorff dimension of the limit set of $M$ is at most $1+\delta$. In particular, $M$ is almost Fuchsian if one chooses $\epsilon<1$.
Conversely, in \cite{zenobiao2} the authors give an estimate of the Hausdorff dimension of the limit set in an almost-Fuchsian manifold $M$ in terms of the maximum of the principal curvatures of the (unique) minimal surface. The degeneration of almost-Fuchsian manifolds is also studied in \cite{sanders1}.

\subsection*{The main steps of the proof}
The proof of Theorem \ref{hyperbolic close to identity} is composed of several steps. 

By using the technique of ``description from infinity'' (see \cite{epstein} and \cite{schkraRenVol}), we construct a foliation $\mathcal{F}$ of $\Hyp^3$ by equidistant surfaces, such that all the leaves of the foliation have the same boundary at infinity, a quasicircle $\Gamma$. By using a theorem proved in \cite{zograf} and \cite[Appendix]{schkraRenVol}, which relates the curvatures of the leaves of the foliation with the Schwarzian derivative of the map which uniformizes the conformal structure of one component of $\partial_\infty \Hyp^3\setminus\Gamma$, we obtain an explicit bound for the distance between two surfaces $F_+$ and $F_-$ of $\mathcal{F}$, where $F_+$ is concave and $F_-$ is convex, in terms of the Bers norm of $\Gamma$. The distance $d_{\Hyp^3}(F_-,F_+)$ goes to $0$ when $\Gamma$ approaches a circle in $\partial_\infty \Hyp^3$.

A fundamental property of a minimal surface $S$ with boundary at infinity a curve $\Gamma$ is that $S$ is contained in the convex hull of $\Gamma$. The surfaces $F_-$ and $F_+$ of the previous step lie outside the convex hull of $\Gamma$, on the two different sides. Hence every point $x$ of $S$ lies on a geodesic segment orthogonal to two planes $P_-$ and $P_+$ (tangent to $F_-$ and $F_+$ respectively) such that $S$ is contained in the region bounded by $P_-$ and $P_+$. The length of such geodesic segment is bounded by the Bers norm of the quasicircle at infinity, in a way which does not depend on the chosen point $x\in S$. 

The next step in the proof is then a Schauder-type estimate. Considering the function $u$, defined on $S$, which is the hyperbolic sine of the distance from the plane $P_-$, it turns out that $u$ solves the equation 
\begin{equation}
\Delta_S u-2u=0\,, \tag{\ref{lap 2u hyp}}
\end{equation} where $\Delta_S$ is the Laplace-Beltrami operator of $S$. We then apply classical theory of linear PDEs, in particular Schauder estimates, to the equation $\eqref{lap 2u hyp}$ in order to prove that
$$||u||_{C^2(\Omega')}\leq C||u||_{C^0(\Omega)}\,,$$
where $\Omega'\subset\!\subset\Omega$ and $u$ is expressed in normal coordinates centered at $x$. Recall that $\Delta_S$ is the Laplace-Beltrami operator, which depends on the surface $S$. In order to have this kind of inequality, it is then necessary to control the coefficients of $\Delta_S$. This is obtained by a compactness argument for conformal harmonic mappings, adapted from \cite{cuschieri}, recalling that minimal discs in $\Hyp^3$ are precisely the image of conformal harmonic mapping from the disc to $\Hyp^3$. However, to ensure that compact sets in the conformal parametrization are comparable to compact sets in normal coordinates, we will first need to prove a uniform bound of the curvature. For this reason we will assume (as in the statement of Theorem \ref{hyperbolic close to identity}) that the minimal discs we consider have boundary at infinity a $K$-quasicircle, with $K\leq K_0$.

The final step is then an explicit estimate of the principal curvatures at $x\in S$, by observing that the shape operator can be expressed in terms of $u$ and the first and second derivatives of $u$. The Schauder estimate above then gives a bound on the principal curvatures just in terms of the supremum of $u$ in a geodesic ball of fixed radius centered at $x$. By using the first step, since $S$ is contained between $P_-$ and the nearby plane $P_+$, we finally get an estimate of the principal curvatures of a minimal embedded disc only in terms of the Bers norm of the quasicircle at infinity. 

All the previous estimates do not depend on the choice of $x\in S$. Hence the following theorem is actually proved.


\begin{thmx} \label{hyperbolic close to identity quantitative}
There exist constants $K_0>1$ and $C>4$ such that
the principal curvatures $\pm\lambda$ of every minimal surface $S$ in $\Hyp^3$ with $\partial_\infty S=\Gamma$ a $K$-quasicircle, with $K\leq K_0$, are bounded by:
\begin{equation} \label{stima finale hyp}
||\lambda||_\infty\leq\frac{C||\Psi||_\mathcal{B}}{\sqrt{1-C||\Psi||_\mathcal{B}^2}}\,,
\end{equation}
\noindent where $\Gamma=\Psi(S^1)$, $\Psi:\widehat\C\to\widehat\C$ is a quasiconformal map, conformal on $\widehat\C\setminus \D$, and $||\Psi||_\mathcal{B}$ denotes the Bers norm of $\Psi$. 
\end{thmx}

Observe that the estimate holds in a neighborhood of the identity (which represents circles in $\partial_\infty \Hyp^3$), in the sense of universal Teichm\"uller space.
Theorem \ref{hyperbolic close to identity} is then a consequence of Theorem \ref{hyperbolic close to identity quantitative}, using the well-known fact that the Bers embedding is locally bi-Lipschitz.

\subsection*{Organization of the paper}

The structure of the paper is as follows. In Section \ref{subsection Hyp AdS}, we introduce the necessary notions on hyperbolic space and some properties of minimal surfaces and convex hulls. In Section \ref{section universal} we introduce the theory of quasiconformal maps and universal Teichm\"uller space.
In Section \ref{section minimal} we prove Theorem \ref{hyperbolic close to identity}. The Section is split in several subsections, containing the steps of the proof. In Section \ref{application} we discuss how Theorem \ref{theorem uniqueness}, Corollary \ref{cor teich dist}, Corollary \ref{corollary teichmuller almost fuchsian} and Corollary \ref{corollary neighborhood} follow from Theorem \ref{hyperbolic close to identity}.

\subsection*{Acknowledgements}

I am very grateful to Jean-Marc Schlenker for his guidance and patience. Most of this work was done during my (very pleasent) stay at University of Luxembourg; I would like to thank the Institution for the hospitality. I am very thankful to my advisor Francesco Bonsante and to Zeno Huang for many interesting discussions and suggestions. 
I would like to thank an anonymous referee for many observations and advices which highly improved the presentation of the paper.


\section{Minimal surfaces in hyperbolic space} \label{subsection Hyp AdS}


We consider (3+1)-dimensional Minkowski space $\R^{3,1}$ as $\R^4$ endowed with the bilinear form 
\begin{equation} \label{bil form 31}
\langle x,y\rangle=x^1 y^1+x^2 y^2+x^3 y^3-x^4 y^4\,.
\end{equation}
 The \emph{hyperboloid model} of hyperbolic 3-space is $$\Hyp^3=\left\{x\in\R^{3,1}:\langle x,x\rangle=-1,x^4>0\right\}.$$ 
The induced metric from $\R^{3,1}$ gives $\Hyp^3$ a Riemannian metric of constant curvature -1. The group of orientation-preserving isometries of $\Hyp^3$ is $\isom(\Hyp^3)\cong \SO_+(3,1)$, namely the group of linear isometries of $\R^{3,1}$ which preserve orientation and do not switch the two connected components of the quadric $\left\{\langle x,x\rangle=-1\right\}$. Geodesics in hyperbolic space are the intersection of $\Hyp^3$ with linear planes $X$ of $\R^{3,1}$ (when nonempty); totally geodesic planes are the intersections with linear hyperplanes and are isometric copies of hyperbolic plane $\Hyp^2$. 


We denote by $d_{\Hyp^3}(\cdot,\cdot)$ the metric on $\Hyp^3$ induced by the Riemannian metric. It is easy to show that \begin{equation} \label{formula distanza iperbolico}
\cosh(d_{\Hyp^3}(p,q))=|\langle p,q\rangle|
\end{equation}
\noindent and other similar formulae which will be used in the paper.

Note that $\Hyp^3$ can also be regarded as the projective domain $$P(\left\{\langle x,x\rangle<0\right\})\subset\R P^3.$$
Let us denote by $\widehat {\dS^3}$ the region
$$\widehat {\dS^3}=\left\{x\in\R^{3,1}:\langle x,x\rangle=1\right\}$$
and we call \emph{de Sitter space} the projectivization of $\widehat {\dS^3}$, 
$$\dS^3=P(\left\{\langle x,x\rangle>0\right\})\subset \R P^3.$$
Totally geodesic planes in hyperbolic space, of the form $P=X\cap\Hyp^3$, are parametrized by the dual points $X^\perp$ in $\dS^3\subset\R P^3$.

In an affine chart $\left\{x_4\neq 0\right\}$ for the projective model of $\Hyp^3$, hyperbolic space is represented as the unit ball $\left\{(x,y,z):x^2+y^2+z^2<1\right\}$, using the affine coordinates $(x,y,z)=(x^1/x^4,x^2/x^4,x^3/x^4)$. This is called the \emph{Klein model}; although in this model the metric of $\Hyp^3$ is not conformal to the Euclidean metric of $\R^3$, the Klein model has the good property that geodesics are straight lines, and totally geodesic planes are intersections of the unit ball with planes of $\R^3$. It is well-known that $\Hyp^3$ has a natural boundary at infinity, $\partial_\infty\Hyp^3 =P(\left\{\langle x,x\rangle=0\right\})$, which is a 2-sphere and is endowed with a natural complex projective structure - and therefore also with a conformal structure.

Given an embedded surface $S$ in $\Hyp^3$, we denote by $\partial_\infty S$ its \emph{asymptotic boundary}, namely, the intersection of the topological closure of $S$ with $\partial_\infty\Hyp^3$.

\subsection{Minimal surfaces}

This paper is mostly concerned with smoothly embedded surfaces in hyperbolic space. Let $\sigma:S\rar\Hyp^3$ be a smooth embedding and let $N$ be a unit normal vector field to the embedded surface $\sigma(S)$. We denote again by $\langle {\cdot},{\cdot}\rangle$ the Riemannian metric of $\Hyp^3$, which is the restriction to the hyperboloid of the bilinear form \eqref{bil form 31} of $\R^{3,1}$; $\nabla$ and $\nabla^S$ are the ambient connection and the Levi-Civita connection of the surface $S$, respectively. The \emph{second fundamental form} of $S$ is defined as
$$\nabla_{\tilde v}\tilde w=\nabla^S_{\tilde v}\tilde w+\II(v,w)N$$
if $\tilde v$ and $\tilde w$ are vector fields extending $v$ and $w$. The \emph{shape operator} is the $(1,1)$-tensor defined as $B(v)=-\nabla_v N$. It satisfies the property
$$\II(v,w)=\langle B(v),w\rangle\,.$$
\begin{defi}
An embedded surface $S$ in $\Hyp^3$ is minimal if $\mathrm{tr}(B)=0$.
\end{defi}

The shape operator is symmetric with respect to the first fundamental form of the surface $S$; hence the condition of minimality amounts to the fact that the principal curvatures (namely, the eigenvalues of $B$) are opposite at every point.

An embedded disc in $\Hyp^3$ is said to be \emph{area minimizing} if any compact subdisc is locally the smallest area surface among all surfaces with the same boundary. It is well-known that area minimizing surfaces are minimal. The problem of existence for minimal surfaces with prescribed curve at infinity was solved by Anderson; see \cite{anderson} for the original source and \cite{survey} for a survey on this topic.

\begin{theorem}[\cite{anderson}]
Given a simple closed curve $\Gamma$ in $\partial_\infty \Hyp^3$, there exists a complete area minimizing embedded disc $S$ with $\partial_\infty S=\Gamma$.
\end{theorem}

 

A key property used in this paper is that minimal surfaces with boundary at infinity a Jordan curve $\Gamma$ are contained in the convex hull of $\Gamma$. Although this fact is known, we prove it here by applying maximum principle to a simple linear PDE describing minimal surfaces.

\begin{defi}
Given a curve $\Gamma$ in $\partial_\infty \Hyp^3$, the convex hull of $\Gamma$, which we denote by $\mathcal{CH}(\Gamma)$, is the intersection of half-spaces bounded by totally geodesic planes $P$ such that $\partial_\infty P$ does not intersect $\Gamma$, and the half-space is taken on the side of $P$ containing $\Gamma$.
\end{defi}

Hereafter $\Hess \,u$ denotes the Hessian of a smooth function $u$ on the surface $S$, i.e. the (1,1) tensor 
$$\Hess \,u(v)=\nabla^S_v \grad u\,.$$
Sometimes the Hessian is also considered as a (2,0) tensor, which we denote (in the rare occurrences) with
$$\nabla^2 u(v,w)=\langle \Hess \,u(v),w\rangle\,.$$
Finally, $\Delta_S$ denotes the Laplace-Beltrami operator of $S$, which can be defined as $$\Delta_S u=\mathrm{tr}(\Hess \,u)\,.$$
Observe that, with this definition, $\Delta_S$ is a negative definite operator.

\begin{prop} \label{formule hess lap hyp}
Given a minimal surface $S\subset\Hyp^3$ and a plane $P$, let $u:S\rar\R$ be the function $u(x)=\sinh d_{\Hyp^3}(x,P)$. Here $d_{\Hyp^3}(x,P)$ is considered as a signed distance from the plane $P$. Let $N$ be the unit normal to $S$, $B=-\nabla N$ the shape operator, and $E$ the identity operatior. Then
\begin{equation}\label{hessian hyp}
\Hess\, u-u\,E=\sqrt{1+u^2-||\grad u||^2}B
\end{equation} 
as a consequence, $u$ satisfies
\begin{equation}\label{lap 2u hyp}
\Delta_S u-2u=0\,. \tag{$\star$}
\end{equation} 
\end{prop}
\begin{proof}
Consider the hyperboloid model for $\Hyp^3$. Let us assume $P$ is the plane dual to the point $p\in\dS^3$, meaning that $P=p^\perp \cap \Hyp^3$. Then $u$ is the restriction to $S$ of the function $U$ defined on $\Hyp^3$:
\begin{equation} \label{definition linear distance function}
U(x)=\sinh d_{\Hyp^3}(x,P)=\langle x,p\rangle\,.
\end{equation}
 Let $N$ be the unit normal vector field to $S$; we compute $\grad u$ by projecting the gradient $\nabla U$ of $U$ to the tangent plane to $S$:
\begin{gather}
\nabla U=p+\langle p,x \rangle x \\
\grad u(x)=p+\langle p,x \rangle x - \langle p,N\rangle N
\end{gather}

\noindent Now $\Hess u(v)=\nabla^{S}_v \grad u$, where $\nabla^{S}$ is the Levi-Civita connection of $S$, namely the projection of the flat connection of $\R^{3,1}$, and so
$$\Hess \,u(x)(v)=\langle p,x\rangle v - \langle p,N\rangle \nabla^S_v N=u(x)v+\langle \nabla U,N\rangle B(v)\,.$$
Moreover, $\nabla U=\grad u+\langle \nabla U,N\rangle N$ and thus
$$\langle \nabla U,N\rangle^2=\langle\nabla U,\nabla U\rangle-||\grad u||^2=1+u^2-||\grad u||^2$$
which proves (\ref{hessian hyp}). By taking the trace, (\ref{lap 2u hyp}) follows.
\end{proof}

\begin{cor} \label{minimal surface contained convex hull}
Let $S$ be a minimal surface in $\Hyp^3$, with $\partial_\infty(S)=\Gamma$ a Jordan curve. Then $S$ is contained in the convex hull $\mathcal{CH}(\Gamma)$.
\end{cor}
\begin{proof}
If $\Gamma$ is a circle, then $S$ is a totally geodesic plane which coincides with the convex hull of $\Gamma$. Hence we can suppose $\Gamma$ is not a circle. Consider a plane $P_-$ which does not intersect $\Gamma$ and the function $u$ defined as in Equation \eqref{definition linear distance function} in Proposition \ref{formule hess lap hyp}, with respect to $P_-$. Suppose their mutual position is such that $u\geq 0$ in the region of $S$ close to the boundary at infinity (i.e. in the complement of a large compact set). If there exists some point where $u<0$, then at a minimum point $\Delta_S u=2u<0$, which gives a contradiction. The proof is analogous for a plane $P_+$ on the other side of $\Gamma$, by switching the signs. Therefore every convex set containing $\Gamma$ contains also $S$.
\end{proof}


\section{Universal Teichm\"uller space} \label{section universal}

The aim of this section is to introduce the theory of quasiconformal mappings and universal Teichm\"uller space. We will give a brief account of the very rich and developed theory. Useful references are \cite{gardiner, gardiner2, ahlforsbook, fletchermarkoviclibro} and the nice survey \cite{sugawa}.

\subsection{Quasiconformal mappings and universal Teichm\"uller space} \label{subsection universal}

We recall the definition of quasiconformal map.

\begin{defi} \label{definitions quasiconformal map}
Given a domain $\Omega\subset\C$, an orientation-preserving homeomorphism $$f:\Omega\to f(\Omega)\subset\C$$ is \emph{quasiconformal} if $f$ is absolutely continuous on lines and there exists a constant $k<1$ such that $$|\partial_{\overline z}f|\leq k |\partial_{z}f|\,.$$
\end{defi}

Let us denote $\mu_f=\partial_{\overline z}f/\partial_{z}f$, which is called \emph{complex dilatation} of $f$. This is well-defined almost everywhere, hence it makes sense to take the $L_\infty$ norm. Thus a homeomorphism $f:\Omega\to f(\Omega)\subset\C$ is quasiconformal if $||\mu_f||_\infty<1$. Moreover, a quasiconformal map as in Definition \ref{definitions quasiconformal map} is called $K$-\emph{quasiconformal}, where
$$K=\frac{1+k}{1-k}\,.$$
It turns out that the best such constant $K\in[1,+\infty)$ represents the \emph{maximal dilatation} of $f$, i.e. the supremum over all $z\in\Omega$ of the ratio between the major axis and the minor axis of the ellipse which is the image of a unit circle under the differential $d_z f$. 

It is known that a $1$-quasiconformal map is conformal, and that the composition of a $K_1$-quasiconformal map and a $K_2$-quasiconformal map is $K_1 K_2$-quasiconformal. Hence composing with conformal maps does not change the maximal dilatation. 

Actually, there is an explicit formula for the complex dilatation of the composition of two quasiconformal maps $f,g$ on $\Omega$:
\begin{equation} \label{composition belt differentials}
\mu_{g\circ f^{-1}}=\frac{\partial_z f}{\overline{\partial_{ z} f}}\frac{\mu_g-\mu_f}{1-\overline{\mu_f}\mu_g}\,.
\end{equation}
Using Equation \eqref{composition belt differentials}, one can see that $f$ and $g$ differ by post-composition with a conformal map if and only if $\mu_f=\mu_g$ almost everywhere. We now mention the classical and important result of existence of quasiconformal maps with given complex dilatation.
\\

\noindent {\bf{Measurable Riemann mapping Theorem}}. Given any measurable function $\mu$ on $\C$ there exists a unique quasiconformal map $f:\C\to\C$ such that $f(0)=0$, $f(1)=1$ and $\mu_f=\mu$ almost everywhere in $\C$.
\\

The uniqueness part of Measurable Riemann mapping Theorem means that every two solutions (which can be thought as maps on the Riemann sphere $\widehat\C$) of the equation $$ (\partial_{z}f) \mu=\partial_{\overline z}f$$
differ by post-composition with a M\"obius transformation of $\widehat \C$. 

Given any fixed $K\geq 1$, $K$-quasiconformal mappings have an important compactness property. See \cite{gardiner} or \cite{lehto}. 

\begin{theorem} \label{Compactness property of quasicircles}
Let $K>1$ and $f_n:\widehat \C\to\widehat\C$ be a sequence of $K$-quasiconformal mappings such that, for three fixed points $z_1,z_2,z_3\in\widehat\C$, the mutual spherical distances are bounded from below: there exists a constant $C_0>0$ such that
$$d_{\mathbb{S}^2}(f_n(z_i),f_n(z_j))>C_0$$
for every $n$ and for every choice of $i,j=1,2,3$, $i\neq j$. Then there exists a subsequence $f_{n_k}$ which converges uniformly to a $K$-quasiconformal map $f_\infty:\widehat\C\to\widehat\C$.
\end{theorem}

\subsection{Quasiconformal deformations of the disc} \label{subsec: quasiconformal model disc}

It turns out that every quasiconformal homeomorphisms of $\D$ to itself extends to the boundary $\partial\D=S^1$. Let us consider the space:
$$QC(\D)=\left\{\Phi:\D\to\D\text { quasiconformal}\right\}/\sim$$
where $\Phi\sim \Phi'$ if and only if $\Phi|_{S^1}=\Phi'|_{S^1}$. Universal Teichm\"uller space is then defined as 
$$\Teich(\D)=QC(\D)/\mbox{M\"ob}(\D)\,,$$
where $\mbox{M\"ob}(\D)$ is the subgroup of M\"obius transformations of $\D$.
Equivalently, $\Teich(\D)$ is the space of quasiconformal homeomorphisms $\Phi:\D\to\D$ which fix $1$, $i$ and $-1$ up to the same relation $\sim$.

Such quasiconformal homeomorphisms of the disc can be obtained in the following way. Given a domain $\Omega$, elements in the unit ball of the (complex-valued) Banach space $L^\infty(\D)$ are called \emph{Beltrami differentials} on $\Omega$. Let us denote this unit ball by:
$$\mathrm{Belt}(\D)=\{\mu\in L^\infty(\D)|\;||\mu||_\infty<1\}\,.$$ Given any $\mu$ in $\mathrm{Belt}(\D)$, let us define  $\hat\mu$ on $\C$ by extending $\mu$ on $\C\setminus \D$ so that
$$\hat\mu(z)=\overline{\mu(1/\overline{z})}\,.$$
The quasiconformal map $f^\mu:\C\to\C$ such that $\mu_{f^\mu}=\hat\mu$ fixing $1$, $i$ and $-1$, whose existence is provided by Measurable Riemann mapping Theorem, maps $\partial\D$ to itself by the uniqueness part. Therefore $f^\mu$ restricts to a quasiconformal homeomorphism of $\D$ to itself.

The Teichm\"uller distance on $\Teich(\D)$ is defined as
$$d_{\Teich(\D)}([\Phi],[\Phi'])=\frac{1}{2}\inf \log K(\Phi_1^{-1}\circ\Phi_1')\,,$$
where the infimum is taken over all quasiconformal maps $\Phi_1\in[\Phi]$ and $\Phi_1'\in[\Phi']$.
It can be shown that $d_{\Teich(\D)}$ is a well-defined distance on  Teichm\"uller space, and $(\Teich(\D),d_{\Teich(\D)})$ is a complete metric space.

\subsection{Quasicircles and Bers embedding} \label{subsec: quasidiscs}

We now want to discuss another interpretation of  Teichm\"uller space, as the space of quasidiscs, and the relation with the Schwartzian derivative and the Bers embedding.

\begin{defi} \label{defi quasicircle}
A \emph{quasicircle} is a simple closed curve $\Gamma$ in $\widehat\C$ such that $\Gamma=\Psi(S^1)$ for a quasiconformal map $\Psi$. Analogously, a \emph{quasidisc} is a domain $\Omega$ in $\widehat\C$ such that $\Omega=\Psi(\D)$ for a quasiconformal map $\Psi:\widehat\C\to\widehat\C$. 
\end{defi}

Let us denote $\D^*=\{z\in\widehat\C:|z|>1\}=\{z\in\C:|z|>1\}\cup\{\infty\}$. We remark that in the definition of quasicircle, it would be equivalent to say that $\Gamma$ is the image of $S^1$ by a $K'$-quasiconformal map of $\widehat \C$ (not necessarily conformal on $\D^*$). However, the maximal dilatation $K'$ might be different, with $K\leq K'\leq 2K$.
Hence we consider the space of quasidiscs:
 $$QD(\D)=\{\Psi:\widehat\C\to\widehat\C:\Psi|_\D\text{ is quasiconformal and }\Psi|_{\D^*}\text{ is conformal}\}/\sim\,,$$
 where the equivalence relation is $\Psi\sim \Psi'$ if and only if $\Psi|_{\D^*}=\Psi'|_{\D^*}$. We will again consider the quotient of $QD(\D)$ by M\"obius transformation. 

Given a Beltrami differential $\mu\in\mathrm{Belt}(\D)$, one can construct a quasiconformal map on $\widehat \C$, by applying Measurable Riemann mapping Theorem to the Beltrami differential obtained by extending $\mu$ to $0$ on 
 $\D^*$. The quasiconformal map obtained in this way (fixing the three points $0$,$1$ and $\infty$) is denoted by $f_\mu$. A well-known lemma (see \cite[§5.4, Lemma 3]{gardiner}) shows that, given two Beltrami differentials $\mu,\mu'\in\mathrm{Belt}(\D)$, $f^\mu|_{S^1}=f^{\mu'}|_{S^1}$ if and only if $f_\mu|_{\D^*}=f_{\mu'}|_{\D^*}$. Using this fact it can be shown that $\Teich(\D)$ is identified to $QD(\D)/\mbox{M\"ob}(\widehat\C)$, or equivalently to the subset of $ QD(\D)$ which fix $0$, $1$ and $\infty$.
  
We will say that a quasicircle $\Gamma$ is a $K$-quasicircle if
$$K=\inf_{\Gamma=\Psi(S^1) \atop \Psi\in Q\!D(\D)} K(\Psi)\,.$$
It is easily seen that the condition that $\Gamma=\Psi(S^1)$ is a $K$-quasicircle is equivalent to the fact that the element $[\Phi]$ of the first model $\Teich(\D)=QC(\D)/\mbox{M\"ob}(\D)$ which corresponds to $[\Psi]$ has Teichm\"uller distance from the identity
$d_{\Teich(\D)}([\Phi],[\mathrm{id}])=(\log K)/2$.

By using the model of quasidiscs for 
 Teichm\"uller space, we now introduce the Bers norm on $\Teich(\D)$.
Recall that, given a holomorphic function  $f:\Omega\to\C$ with $f'\neq 0$ in $\Omega$, the \emph{Schwarzian derivative} of $f$ is the holomorphic function
$$S_f=\left(\frac{f''}{f'}\right)'-\frac{1}{2} \left(\frac{f''}{f'}\right)^2\,.$$
It can be easily checked that $S_{1/f}=S_f$, hence the Schwarzian derivative can be defined also for meromorphic functions at simple poles. The Schwarzian derivative vanishes precisely on M\"obius transformations. 

   Let us now consider the space of holomorphic quadratic differentials on $\D$. We will consider the following norm, for a holomorphic quadratic differential $q=h(z)dz^2$:
$$||q||_{\infty}=\sup_{z\in\D}e^{-2\eta(z)}|h(z)|\,,$$ 
where $e^{2\eta(z)}|dz|^2$ is the Poincar\'e metric of constant curvature $-1$ on $\D$. Observe that $||q||_\infty$ behaves like a function, in the sense that it is invariant by pre-composition with M\"obius transformations of $\D$, which are isometries for the Poincar\'e metric. 

We now define the \emph{Bers embedding} of universal Teichm\"uller space.
This is the map
$\beta_\D$ which associates to $[\Psi]\in\Teich(\D)=QD(\D)/\mbox{M\"ob}(\widehat\C)$ the Schwarzian derivative $S_\Psi$. Let us denote by $||\bullets||_{\mathcal{Q}(\D^*)}$ the norm on holomorphic quadratic differentials on $\D^*$ obtained from the $||\bullets||_\infty$ norm on $\D$, by  identifying $\D$ with $\D^*$ by an inversion in $S^1$. Then
$$\beta_\D:\Teich(\D)\rightarrow \mathcal{Q}(\D^*)$$
 is an embedding of $\Teich(\D)$
 in the Banach space $(\mathcal{Q}(\D^*),||\bullets||_{\mathcal{Q}(\D^*)})$ of bounded holomorphic quadratic differentials (i.e. for which $||q||_{\mathcal{Q}(\D^*)}<+\infty$). Finally, the Bers norm of en element $\Psi\in\Teich(\D)$ is 
 $$||\Psi||_{\mathcal B}=||\beta_\D[\Psi]||_\infty=||S_\Psi||_{\mathcal{Q}(\D^*)}\,.$$

The fact that the Bers embedding is locally bi-Lipschitz will be used in the following. See for instance \cite[Theorem 4.3]{fletmarkkahn}. In the statement, we again implicitly identify the models of universal Teichm\"uller space by quasiconformal homeomorphisms of the disc (denoted by $[\Phi]$) and by quasicircles (denoted by $[\Psi]$).

\begin{theorem} \label{thm bers emb locally bilipschitz}
Let $r>0$. There exist constants $b_1$ and $b_2=b_2(r)$ such that, for every $[\Psi],[\Psi']$ in the ball of radius $r$ for the Teichm\"uller distance centered at the origin 
(i.e. $d_{\Teich(\D)}([\Psi],[\mathrm{id}]),d_{\Teich(\D)}([\Psi'],[\mathrm{id}])<r$),
$$b_1 ||\beta_\D[\Psi]-\beta_\D[\Psi]||_\infty\leq d_{\Teich(\D)}([\Psi],[\Psi'])\leq b_2 ||\beta_\D[\Psi]-\beta_\D[\Psi]||_\infty\,.$$
\end{theorem}

We conclude this preliminary part by mentioning a theorem by Nehari, see for instance \cite{lehto} or \cite{fletchermarkoviclibro}.
\\

\noindent {\bf{Nehari Theorem}}.
The image of the Bers embedding is contained in the ball of radius $3/2$ in $(\mathcal{Q}(\D^*),||\bullets||_{\mathcal{Q}(\D^*)})$, and contains the ball of radius $1/2$.


\section{Minimal surfaces in $\Hyp^3$} \label{section minimal}

The goal of this section is to prove Theorem \ref{hyperbolic close to identity}. The proof is divided into several steps, whose general idea is the following:
\begin{enumerate}
\item \label{step1} Given $\Psi\in QD(\D)$, if $||\Psi||_{\mathcal B}$ is small, then there is a foliation $\mathcal{F}$ of a convex subset $\mathcal C$ of $\Hyp^3$ by equidistant surfaces. All the surfaces $F$ of $\mathcal{F}$ have asymptotic boundary the quasicircle $\Gamma=\Psi(S^1)$. Hence the convex hull of $\Gamma$ is trapped between two parallel surfaces, whose distance is estimated in terms of $||\Psi||_{\mathcal B}$.
\item \label{step2} As a consequence of point (\ref{step1}), given a minimal surface $S$ in $\Hyp^3$ with $\partial_\infty(S)=\Gamma$, for every point $x\in S$ there is a geodesic segment through $x$ of small length orthogonal at the endpoints to two planes $P_-$,$P_+$ which do not intersect $\mathcal{C}$. Moreover $S$ is contained between $P_-$ and $P_+$.
\item Since $S$ is contained between two parallel planes close to $x$, the principal curvatures of $S$ in a neighborhood of $x$ cannot be too large. In particular, we use Schauder theory to show that the principal curvatures of $S$ at a point $x$ are uniformly bounded in terms of the distance from $P_-$ of points in a neighborhood of $x$.
\item Finally, the distance from $P_-$ of points of $S$ in a neighborhood of $x$  is estimated in terms of the distance of points in $P_+$ from $P_-$, hence is bounded in terms of the Bers norm $||\Psi||_{\mathcal B}$.
\end{enumerate}
It is important to remark that the estimates we give are uniform, in the sense that they do not depend on the point $x$ or on the surface $S$, but just on the Bers norm of the quasicircle at infinity. The above heuristic arguments are formalized in the following subsections.

\subsection{Description from infinity} \label{infinity}
The main result of this part is the following. See Figure \ref{parallel_planes}.

\begin{prop} \label{computation width bers norm}
Let $A<1/2$. Given an embedded minimal disc $S$ in $\Hyp^3$ with boundary at infinity a quasicircle $\partial_\infty S=\Psi(S^1)$ with $||\Psi||_{\mathcal{B}}\leq A$, every point of $S$ lies on a geodesic segment of length at most
$\arctanh(2A)$ orthogonal at the endpoints to two planes $P_-$ and $P_+$, such that the convex hull $\mathcal{CH}(\Gamma)$ is contained between $P_-$ and $P_+$.
\end{prop}

\begin{figure}[htbp]
\centering
\includegraphics[height=6cm]{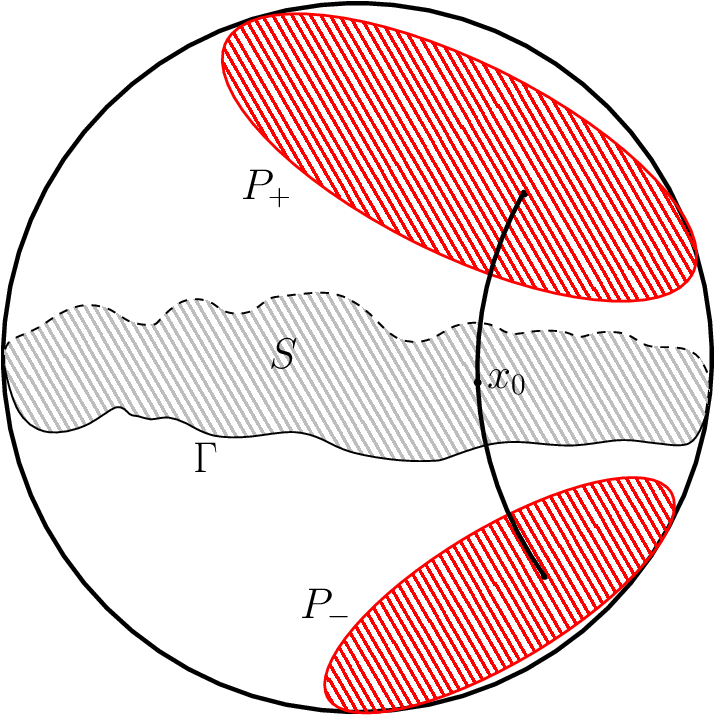}
\caption{The statement of Proposition \ref{computation width bers norm}. The geodesic segment through $x_0$ has length $\leq w$, for $w=\arctanh(2||\Psi||_{\mathcal{B}})$, and this does not depend on $x_0\in S$.} \label{parallel_planes}
\end{figure}

\begin{remark} \label{remark width}
A consequence of Proposition \ref{computation width bers norm} is that the Hausdorff distance between the two boundary components of $\mathcal{CH}(\Gamma)$ is bounded by $\arctanh(2||\Psi||_{\mathcal{B}})$. Hence it would be natural to try to define in such a way a notion of \emph{thickness} or \emph{width} of the convex hull:
$$w(\Gamma)=\max\{\inf_{x\in\partial_- \mathcal{CH}(\Gamma)}d(x,\partial_+ \mathcal{CH}(\Gamma)), \inf_{x\in\partial_+ \mathcal{CH}(\Gamma)}d(x,\partial_- \mathcal{CH}(\Gamma))\}$$
 However, a bound on the Hausdorff distance is not sufficient for the purpose of this paper. It will become clear in the proof of Theorem \ref{hyperbolic close to identity quantitative} and Theorem \ref{hyperbolic close to identity}, and in particular for the application of Lemma \ref{lemma distance formulae}, that the necessary property is the existence of two support planes which are \emph{both} orthogonal to a geodesic segment of short length through any point $x_0\in S$.
\end{remark}

We review here some important facts on the so-called description from infinity of surfaces in hyperbolic space. For details, see \cite{epstein} and \cite{schkraRenVol}. Given an embedded surface $S$ in $\Hyp^3$ with bounded principal curvatures, let $I$ be its first fundamental form and $\II$ the second fundamental form. Recall we defined $B=-\nabla N$ its shape operator, for $N$ the oriented unit normal vector field (we fix the convention that $N$ points towards the $x_4>0$ direction in $\R^{3,1}$), so that $\II=I(B\cdot,\cdot)$. Denote by $E$ the identity operator. Let $S_\rho$ be the $\rho$-equidistant surface from $S$ (where the sign of $\rho$ agrees with the choice of unit normal vector field to $S$). For small $\rho$, there is a map from $S$ to $S_\rho$ obtained following the geodesics orthogonal to $S$ at every point.

\begin{lemma} \label{lemma equidistant surfaces hyp}
Given a smooth surface $S$ in $\Hyp^3$, let $S_\rho$ be the surface at distance $\rho$ from $S$, obtained by following the normal flow at time $\rho$. Then the pull-back to $S$ of the induced metric on the surface $S_\rho$ is given by:
\begin{equation}
I_\rho=I((\cosh(\rho) E-\sinh(\rho) B)\cdot,(\cosh(\rho) E-\sinh(\rho) B)\cdot)\,.
\end{equation}
\noindent The second fundamental form and the shape operator of $S_\rho$ are given by
\begin{gather}
\II_\rho=I((-\sinh(\rho) E+\cosh(\rho) B)\cdot,(\cosh(\rho) E-\sinh(\rho) B)\cdot) \\
B_\rho=(\cosh(\rho) E-\sinh(\rho) B)^{-1}(-\sinh(\rho) E+\cosh(\rho) B)\,.
\end{gather}
\end{lemma}
\begin{proof}
In the hyperboloid model, let $\sigma:\D\rar\Hyp^2$ be the minimal embedding of the surface $S$, with oriented unit normal $N$. The geodesics orthogonal to $S$ at a point $x$ can be written as
$$\gamma_x(\rho)=\cosh(\rho)\sigma(x)+\sinh(\rho)N(x)\,.$$
Then we compute
\begin{align*}
I_\rho(v,w)=&\langle d\gamma_x(\rho)(v),d\gamma_x(\rho)(w) \rangle \\
=&\langle \cosh(\rho)d\sigma_x(v)+\sinh(\rho)dN_x(v),\cosh(r)d\sigma_x(w)+\sinh(\rho)dN_x(w)\rangle \\
=&I(\cosh(\rho)v-\sinh(\rho)B(v),\cosh(\rho)w-\sinh(\rho)B(w))\,.
\end{align*}
The formula for the second fundamental form follows from the fact that $\II_\rho=-\frac{1}{2}\frac{dI_\rho}{d\rho}$.
\end{proof}

It follows that, if the principal curvatures of a minimal surface $S$ are $\lambda$ and $-\lambda$, then the principal curvatures of $S_\rho$ are 
\begin{equation} \label{formulae prin curv equi} 
\lambda_\rho=\frac{\lambda-\tanh(\rho)}{1-\lambda\tanh(\rho)}\qquad\lambda'_\rho=\frac{-\lambda-\tanh(\rho)}{1+\lambda\tanh(\rho)}\,.
\end{equation}
 In particular, if $-1\leq\lambda\leq 1$, then $I_\rho$ is a non-singular metric for every $\rho$. The surfaces $S_\rho$ foliate $\Hyp^3$ and they all have asymptotic boundary $\partial_\infty S_\rho=\partial_\infty S$.

We now define the first, second and third fundamental form at infinity associated to $S$. Recall the second and third fundamental form of $S$ are $\II=I(B\cdot,\cdot)$ and $\III=I(B\cdot,B\cdot)$.
\begin{gather}
I^*=\lim_{\rho\rar\infty}2e^{-2\rho}I_\rho=\frac{1}{2}I((E-B)\cdot,(E-B)\cdot)=\frac{1}{2}(I-2\II+\III) \\
B^*=(E-B)^{-1}(E+B) \\
\II^*=\frac{1}{2}I((E+B)\cdot,(E-B)\cdot)=I^*(B^*\cdot,\cdot) \\
\III^*=I^*(B^*\cdot,B^*\cdot)
\end{gather}
We observe that the metric $I_\rho$ and the second fundamental form can be recovered as
\begin{gather}
I_\rho=\frac{1}{2}e^{2\rho}I^*+\II^*+\frac{1}{2}e^{-2\rho}\III^* \label{first fundamental form formula}\\
\II_\rho=-\frac{1}{2}\frac{dI_\rho}{d\rho}=\frac{1}{2}I^*((e^\rho E+e^{-\rho} B^*)\cdot,(-e^\rho E+e^{-\rho} B^*)\cdot) \\
B_\rho=(e^\rho E+e^{-\rho} B^*)^{-1} (-e^\rho E+e^{-\rho} B^*) \label{shape operator formula}
\end{gather}
The following relation can be proved by some easy computation:
\begin{lemma}[{\cite[Remark 5.4 and 5.5]{schkraRenVol}}] \label{gauss codazzi infinity}
The embedding data at infinity $(I^*,B^*)$ associated to an embedded surface $S$ in $\Hyp^3$ satisfy the equation
\begin{equation} \label{gauss infinity}
tr(B^*)=-K_{I^*}\,,
\end{equation}
 where $K_{I^*}$ is the curvature of $I^*$. Moreover, $B^*$ satisfies the Codazzi equation with respect to $I^*$:
 \begin{equation} \label{codazzi infinity}
 d^{\nabla_{\! I^*}}\!B^*=0\,.
 \end{equation}
\end{lemma}

A partial converse of this fact, which can be regarded as a fundamental theorem from infinity, is the following theorem. This follows again by the results in \cite{schkraRenVol}, although it is not stated in full generality here.
\begin{theorem} \label{fundamental theorem infinity}
Given a Jordan curve $\Gamma\subset\partial_\infty \Hyp^3$, let $(I^*,B^*)$ be a pair of a metric in the conformal class of a connected component of $\partial_\infty \Hyp^3\setminus\Gamma$ and a self-adjoint $(1,1)$-tensor, satisfiying the conditions \eqref{gauss infinity} and \eqref{codazzi infinity} as in Lemma \ref{gauss codazzi infinity}. Assume the eigenvalues of $B^*$ are positive at every point. Then there exists a foliation of $\Hyp^3$ by equidistant surfaces $S_\rho$, for which the first fundamental form at infinity (with respect to $S=S_0$) is $I^*$ and the shape operator at infinity is $B^*$.
\end{theorem}

We want to give a relation between the Bers norm of the quasicircle $\Gamma$ and the existence of a foliation of $\Hyp^3$ by equidistant surfaces with boundary $\Gamma$, containing both convex and concave surfaces. We identify $\partial_\infty \Hyp^3$ to $\widehat\C$ by means of the stereographic projection, so that $\D$ correponds to the lower hemisphere of the sphere at infinity. The following property will be used, see \cite{zograf} or \cite[Appendix A]{schkraRenVol}.

\begin{theorem} \label{teorema real part schwartzian}
Let $\Gamma\subset \partial_\infty\Hyp^3$ be a Jordan curve. If $I^*$ is the complete hyperbolic metric in the conformal class of a connected component $\Omega$ of $\partial_\infty\Hyp^3\setminus\Gamma$, and $\II_0^*$ is the traceless part of the second fundamental form at infinity $\II^*$, then $-\II_0^*$ is the real part of the Schwarzian derivative of the isometry $\Psi:\D^* \rar\Omega$, namely the map $\Psi$ which uniformizes the conformal structure of $\Omega$:
\begin{equation} \label{thm TZ SK}
\II_0^*=-Re(S_\Psi)\,.
\end{equation}

\end{theorem}

We now derive, by straightforward computation, a useful relation. 

\begin{lemma} \label{determinant and bers norm} 
Let $\Gamma=\Psi(S^1)$ be a quasicircle, for $\Psi\in QD(\D)$. If $I^*$ is the complete hyperbolic metric in the conformal class of a connected component $\Omega$ of $\partial_\infty \Hyp^3\setminus\Gamma$, and $B_0^*$ is the traceless part of the shape operator at infinity $B^*$, then
\begin{equation} \label{relation det B0 bers norm}
\sup_{z\in\Omega}|\det B_0^*(z)|=||\Psi||^2_{\mathcal{B}}\,.
\end{equation} 
\end{lemma}
\begin{proof}
From Theorem \ref{teorema real part schwartzian}, $B^*_0$ is the real part of the holomorphic quadratic differential $-S_\Psi$. In complex conformal coordinates, we can assume that
$$I^*=e^{2\eta}|dz|^2=\begin{pmatrix} 0 & \frac{1}{2}e^{2\eta} \\ \frac{1}{2}e^{2\eta} & 0 \end{pmatrix}$$
and $S_\Psi=h(z)dz^2$, so that
$$\II_0^*=-\frac{1}{2}(h(z)dz^2+\overline {h(z)}d\bar z^2)=-\begin{pmatrix} \frac{1}{2}h & 0 \\ 0 & \frac{1}{2}\bar h \end{pmatrix}$$
and finally
$$B_0^*=(I^*)^{-1}\II_0^*=-\begin{pmatrix} 0 & e^{-2\eta}\bar h \\ e^{-2\eta} h & 0 \end{pmatrix}.$$
Therefore $|\det B_0^*(z)|=e^{-4\eta(z)}|h(z)|^2$. Moreover, by definition of Bers embedding, $\mathcal{B}([{\Psi}])=S_\Psi$, because $\Psi$ is a holomorphic map from $\D^*$ which maps $S^1=\partial\D$ to $\Gamma$. Since 
$$||\Psi||_{\mathcal{B}}^2=\sup_{z\in\Omega}(e^{-4\eta(z)}|h(z)|^2)\,,$$
 this concludes the proof.
\end{proof}

We are finally ready to prove Proposition \ref{computation width bers norm}.
\begin{proof}[Proof of Proposition \ref{computation width bers norm}]
Suppose again $I^*$ is a hyperbolic metric in the conformal class of $\Omega$. Since $tr(B^*)=1$ by Lemma \ref{gauss codazzi infinity}, we can write $B^*=B^*_0+(1/2)E$, where $B^*_0$ is the traceless part of $B^*$. 
The symmetric operator $B^*$ is diagonalizable; therefore we can suppose its eigenvalues at every point are $(a+1/2)$ and $(-a+1/2)$, where $a$ is a positive number depending on the point. Hence $\pm a$ are the eigenvalues of the traceless part $B_0^*$.

By using Equation \eqref{relation det B0 bers norm} of Lemma \ref{determinant and bers norm}, and observing that $|\det B_0^*|=a^2$, one obtains $||\Psi||_{\mathcal{B}}=||a||_\infty$. Since this quantity is less than $A<1/2$ by hypothesis, at every point $a<1/2$, and therefore the eigenvalues of $B^*$ are positive at every point.

By Theorem \ref{fundamental theorem infinity} there exists a smooth foliation $\mathcal{F}$ of $\Hyp^3$ by equidistant surfaces $S_\rho$, whose first fundamental form and shape operator are as in equations (\ref{first fundamental form formula}) and (\ref{shape operator formula}) above. We are going to compute  
$$\rho_1=\inf\left\{\rho:B_\rho\text{ is non-singular and negative definite}\right\}$$ and $$\rho_2=\sup\left\{\rho:B_\rho\text{ is non-singular and positive definite}\right\}.$$ Hence $S_{\rho_1}$ is concave and $S_{\rho_2}$ is convex. By Corollary \ref{minimal surface contained convex hull}, $S$ is contained in the region bounded by $S_{\rho_1}$ and $S_{\rho_2}$. We are therefore going to compute $\rho_1-\rho_2$. From the expression (\ref{shape operator formula}), the eigenvalues of $B_\rho$ are
$$\lambda_\rho=\frac{-2e^{2\rho}+(2a+1)}{2e^{2\rho}+(2a+1)}$$
and 
$$\lambda'_\rho=\frac{-2e^{2\rho}+(1-2a)}{2e^{2\rho}+(1-2a)}.$$
Since $a<1/2$, the denominators of $\lambda_\rho$ and $\lambda'_\rho$ are always positive; one has $\lambda_\rho<0$ if and only if $e^{2\rho}>a+1/2$, whereas $\lambda'_\rho<0$ if and only if $e^{2\rho}>-a+1/2$.  Therefore $$\rho_1-\rho_2=\frac{1}{2}\left(\log\left(A+\frac{1}{2}\right)-\log\left(-A+\frac{1}{2}\right)\right)=\frac{1}{2}\log\left(\frac{1+2A}{1-2A}\right)=\arctanh(2A)\,.$$
This shows that every point $x$ on $S$ lies on a geodesic orthogonal to the leaves of the foliation, and the distance between the concave surface $S_{\rho_1}$ and the convex surface $S_{\rho_2}$, on the two sides of $x$, is less than $\arctanh(2A)$. Taking $P_-$ and $P_+$ the planes tangent to $S_{\rho_1}$ and $S_{\rho_2}$, the claim is proved.
\end{proof}

\begin{remark}
The proof relies on the observation - given in \cite{schkraRenVol} and expressed here implicitly in Theorem \ref{fundamental theorem infinity} - that if the shape operator at infinity $B^*$ is positive definite, then one reconstructs the shape operator $B_{\rho}$ as in Equation \eqref{shape operator formula}, and for $\rho=0$ the principal curvatures are in $(-1,1)$. Hence from our argument it follows that, if the Bers norm $||\Psi||_{\mathcal{B}}$ is less than $1/2$, then one finds a surface $S$ with $\partial_\infty S=\Psi(S^1)$, with principal curvatures in $(-1,1)$. This is a special case of the results in \cite{epsteingauss}, where the existence of such surface is used to prove (using techniques of hyperbolic geometry) a generalization of the univalence criterion of Nehari.
\end{remark}



\subsection{Boundedness of curvature} \label{subsec boundedness}

Recall that the curvature of a minimal surface $S$ is given by $K_{S}=-1-\lambda^2$, where $\pm\lambda$ are the principal curvatures of $S$. 
We will need to show that the curvature of a complete minimal surface $S$ is also bounded below in a uniform way, depending only on the complexity of $\partial_\infty S$. This is the content of Lemma \ref{curvature bounded below}.

We will use a conformal identification of $S$ with $\D$. Under this identification the metric takes the form $g_S=e^{2f}|dz|^2$, $|dz|^2$ being the Euclidean metric on $\D$. The following uniform bounds on $f$ are known (see \cite{ahlfors}).
\begin{lemma} \label{bounded conformal factor}
Let $g=e^{2f}|dz|^2$ be a conformal metric on $\D$. Suppose the curvature of $g$ is bounded above, $K_g<-\epsilon^2<0$. Then
\begin{equation} \label{comparison conformal factors}
e^{2f}<\frac{4}{\epsilon^2(1-|z|^2)^2}\,.
\end{equation}
Analogously, if $-\delta^2<K_g$, then
\begin{equation} \label{comparison conformal factors below}
e^{2f}>\frac{4}{\delta^2(1-|z|^2)^2}\,.
\end{equation}
\end{lemma}

\begin{remark} \label{remark compare conformal and geodesic balls}
A consequence of Lemma \ref{bounded conformal factor} is that, for a conformal metric $g=e^{2f}|dz|^2$ on $\D$, if the curvature of $g$ is bounded from above by $K_g<-\epsilon^2<0$, then a conformal ball $B_0(p,R)$ (i.e. a ball of radius $R$ for the Euclidean metric $|dz|^2$) is contained in the geodesic ball of radius $R'$ (for the metric $g$) centered at the same point, where $R'$ only depends from $R$. This can be checked by a simple integration argument, and $R'$ is actually obtained by multiplying $R$ for the square root of the constant in the RHS of Equation \eqref{comparison conformal factors}. Analogously, a lower bound on the curvature, of the form $-\delta^2<K_g$, ensures that the geodesic ball of radius $R$ centered at $p$ is contained in the conformal ball $B_0(p,R')$, where $R'$ depends on $R$ and $\delta$.
\end{remark}

\begin{lemma} \label{curvature bounded below}
For every $K_0>1$, there exists a constant $\Lambda_0>0$ such that all minimal surfaces $S$ with $\partial_\infty S$ a $K$-quasicircle, $K\leq K_0$, have principal curvatures bounded by $||\lambda||_\infty<\Lambda_0$.
\end{lemma}


We will prove Lemma \ref{curvature bounded below} by giving a compactness argument. It is known that a conformal embedding $\sigma:\D\rar \Hyp^3$ is harmonic if and only if $\sigma(\D)$ is a minimal surface, see \cite{eells}. The following Lemma is proved in \cite{cuschieri} in the more general case of CMC surfaces. We give a sketch of the proof here for convenience of the reader.
\begin{lemma} \label{compactness harmonic maps}
Let $\sigma_n:\D\rar\Hyp^3$ a sequence of conformal harmonic maps such that $\sigma(0)=x_0$ and $\partial_\infty(\sigma_n(\D))=\Gamma_n$ is a Jordan curve, and assume $\Gamma_n\rar \Gamma$ in the Hausdoff topology. Then there exists a subsequence $\sigma_{n_k}$ which converges $C^\infty$ on compact subsets to a conformal harmonic map $\sigma_\infty:\D\to\Hyp^3$ with $\partial_\infty(\sigma_\infty(\D))=\Gamma$.
\end{lemma}
\begin{proof}[Sketch of proof]
Consider the coordinates on $\Hyp^3$ given by the Poincar\'e model, namely $\Hyp^3$ is the unit ball in $\R^3$. 
Let $\sigma_n^l$, for $l=1,2,3$, be the components of $\sigma_n$ in such coordinates. Fix $R>0$ for the moment.

Since the curvature of the minimal surfaces $\sigma_n(\D)$ is less than $-1$, from Lemma \ref{bounded conformal factor} (setting $\epsilon=1$) and Remark \ref{remark compare conformal and geodesic balls}, for every $n$ we have that $\sigma_n(B_0(0,2R))$ is contained in a geodesic ball for the induced metric of fixed radius $R'$ centered at $x_0$. In turn, the geodesic ball for the induced metric is clearly contained in the ball $B_{\Hyp^3}(x_0,R')$, for the hyperbolic metric of $\Hyp^3$. We remark that the radius $R'$ only depends on $R$.

We will apply standard Schauder theory (compare also similar applications in Sections \ref{subsection schauder}) to the harmonicity condition
\begin{equation} \label{harmonicity condition}
\Delta_0 \sigma_n^l=-\left(\Gamma_{jk}^l\circ\sigma\right)\left(\frac{\partial \sigma_i^j}{\partial x^1}\frac{\partial \sigma_i^k}{\partial x^1}+\frac{\partial \sigma_i^j}{\partial x^2}\frac{\partial \sigma_i^k}{\partial x^2}\right)=:h_n^l
\end{equation}
for the Euclidean Laplace operator $\Delta_0$, where $\Gamma_{jk}^l$ are the Christoffel symbols of the hyperbolic metric in the Poincar\'e model. 

The RHS in Equation (\ref{harmonicity condition}), which is denoted by $h_n^l$, is uniformly bounded on $B_0(0,2R)$. Indeed Christoffel symbols are uniformly bounded, since $\sigma_n(B_0(0,2R))$ is contained in a compact subset of $\Hyp^3$, as already remarked. The partial derivatives of $\sigma_n^l$ are bounded too, since one can observe that, if the induced metric on $S$ is $e^{2f}|dz|^2$, then $2e^{2f}=||d\sigma||^2$, where
$$||d\sigma||^2=\frac{4}{(1-\Sigma_i (\sigma_n^i)^2)^2}\left(\left( \frac{\partial \sigma_n^1}{\partial x}\right)^2\!\!+\!\left( \frac{\partial \sigma_n^2}{\partial x}\right)^2\!\!+\!\left( \frac{\partial \sigma_n^3}{\partial x}\right)^2\!\!+\!\left( \frac{\partial \sigma_n^1}{\partial y}\right)^2\!\!+\!\left( \frac{\partial \sigma_n^2}{\partial y}\right)^2\!\!+\!\left( \frac{\partial \sigma_n^3}{\partial y}\right)^2 \right)\,.$$
Hence from Lemma \ref{bounded conformal factor} and again the fact that $\sigma_n(B_0(0,2R))$ is contained in a compact subset of $\Hyp^3$, all partial derivatives of $\sigma_n$ are uniformly bounded.

The Schauder estimate for the equation $\Delta_0 \sigma_n^l=h_n^l$ (\cite{giltrud}) give (for every $\alpha\in(0,1)$) a constant $C_1$ such that:
$$|| \sigma_n^l||_{C^{1,\alpha}(B_0(0,R_1))}\leq C_1\left( || \sigma_n^l ||_{C^{0}(B_0(0,2R))}+||h_n^l||_{C^{0}(B_0(0,2R))}  \right)\,.$$
Hence one obtains uniform $C^{1,\alpha}(B_0(0,R_1))$ bounds on $\sigma_n^l$, where $R<R_1<2R$, and this provides $C^{0,\alpha}(B_0(0,R_1))$ bounds on $h_n^l$. Then the following estimate of Schauder-type
$$|| \sigma_n^l||_{C^{2,\alpha}(B_0(0,R_2))}\leq C_2\left( || \sigma_n^l ||_{C^{0}(B_0(0,R_1))}+||h_n^l||_{C^{1,\alpha}(B_0(0,R_1))}  \right)$$
provide $C^{2,\alpha}$ bounds on $B_0(0,R_2)$, for $R<R_2<R_1$. By a boot-strap argument which repeats this construction, uniform $C^{k,\alpha}(B_0(0,R))$ for $\sigma_n^l$ are obtained for every $k$.

 By Ascoli-Arzel\`a theorem, one can extract a subsequence of $\sigma_n$ converging uniformly in $C^{k,\alpha}(B_0(0,R))$ for every $k$. By applying a diagonal procedure one can find a subsequence converging $C^\infty$. One concludes the proof by a diagonal process again on a sequence of compact subsets $B_0(0,R_n)$ which exhausts $\D$.
 
The limit function $\sigma_\infty:\D\to\Hyp^3$ is conformal and harmonic, and thus gives a parametrization of a minimal surface. It remains to show that $\partial_\infty(\sigma_\infty(\D))=\Gamma$. Since each $\sigma_n(\D)$ is contained in the convex hull of $\Gamma_n$, the Hausdorff convergence on the boundary at infinity ensures that $\sigma_\infty(\D)$ is contained in the convex hull of $\Gamma$, and thus $\partial_\infty(\sigma_\infty(\D))\subseteq \Gamma$.

For the other inclusion, assume there exists a point $p\in\Gamma$ which is not in the boundary at infinity of $\sigma_\infty(\D)$. Then there is a neighborhood of $p$ which does not intersect $\sigma_\infty(\D)$, and one can find a totally geodesic plane $P$ such that a half-space bounded by $P$ intersects $\Gamma$ (in $p$, for instance), but does not intersect $\sigma_\infty(\D)$. But such half-space intersects $\sigma_n (\D)$ for large $n$ and this gives a contradiction.
\end{proof}

\begin{proof}[Proof of Lemma \ref{curvature bounded below}]
We argue by contradiction. Suppose there exists a sequence of minimal surfaces $S_n$ bounded by $K$-quasicircles $\Gamma_n$, with $K\leq K_0$, with curvature in a point $K_{S_n}(x_n)\leq -n$. Let us consider isometries $T_n$ of $\Hyp^3$, so that $T_n(x_n)=x_0$. 




We claim that, since the point $x_0$ is contained in the convex hull of $T_n(\Gamma_n)$ for every $n$, the quasicircles $T_n(\Gamma_n)$ can be assumed to be the image of $S^1$ under $K_0$-quasiconformal maps $\Psi_n:\widehat\C\to\widehat \C$, such that $\Psi_n$ maps three points of $S^1$ (say $1$, $i$ and $-1$) to points of $T_n(\Gamma_n)$  at uniformly positive distance from one another in the spherical metric (thus satisfying the hypothesis of Theorem \ref{Compactness property of quasicircles}). Indeed, recall that composing a $K_0$-quasiconformal map by a conformal map does not change the constant $K_0$. Thus it suffices to prove that the quasicircles $T_n(\Gamma_n)=\Psi_n(S^1)$ ($\Psi_n$ a $K_0$-quasiconformal map) contain three points $u_n,v_n,w_n$ at uniformly positive distance from one another, and then one can re-parameterize the quasicircle by pre-composing $\Psi_n$ with a biholomorphism of $\widehat\C$ (which is determined by the image of three points on $S^1$) so that $1,i,-1$ are mapped to $u_n,v_n,w_n$. Moreover, it suffices to prove that the quasicircles $T_n(\Gamma_n)$ contain two points $u_n,v_n$ with distance $d_{S^2}(u_n,v_n)>2C$, where $C$ is some constant independent from $n$. Indeed, the Jordan curve $T_n(\Gamma_n)$ will then necessarily contain a third point $w_n$ such that $d_{S^2}(u_n,w_n)$ and $d_{S^2}(v_n,w_n)$ are larger than $C$. The latter claim is easily proved by contradiction: if the statement was not true, then for every integer $j$ there would exists a quasicircle $T_{n_j}(\Gamma_{n_j})$ which is contained in a ball of radius $1/j$ for the spherical metric on $S^2$. But then it is clear that, for large $j$, the convex hull of $T_{n_j}(\Gamma_{n_j})$ would not contain the fixed point $x_0$. See Figure \ref{contradiction}.

\begin{figure}[htbp]
\centering
\includegraphics[height=6.5cm]{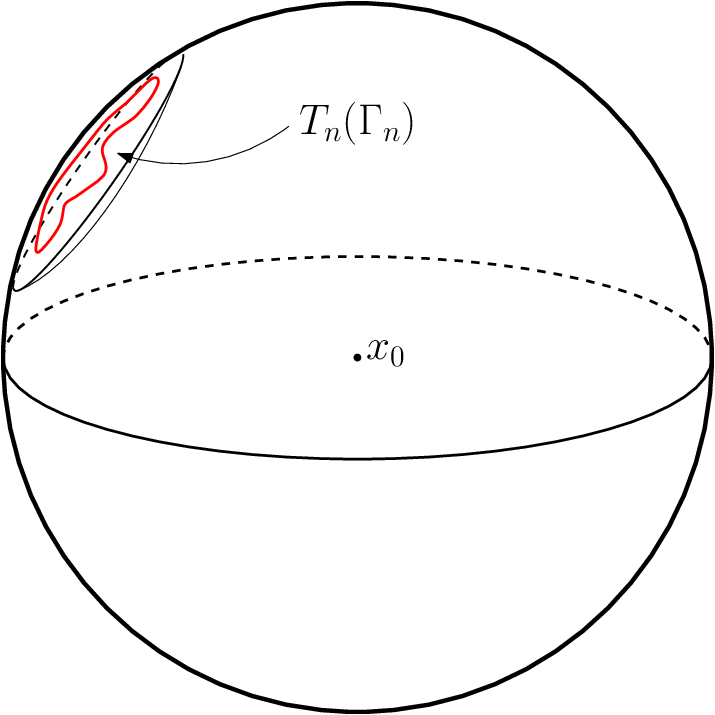}
\caption{If the quasicircle $T_n(\Gamma_n)$ is contained in a small ball for the spherical metric, then the (fixed) point $x_0$ cannot be in the convex hull of the quasicircle.} \label{contradiction}
\end{figure}

By the compactness property in Theorem \ref{Compactness property of quasicircles}, there exists a subsequence $T_{n_k}(\Gamma_{n_k})$ converging to a $K$-quasicircle $\Gamma_\infty$, with $K\leq K_0$. By Lemma \ref{compactness harmonic maps}, the minimal surfaces $T_{n_k}(S_{n_k})$ converge $C^\infty$ on compact subsets (up to a subsequence) to a smooth minimal surface $S_\infty$ with $\partial_\infty(S_\infty)=\Gamma_\infty$. Hence the curvature of $T_{n_k}(S_{n_k})$ at the point $x_0$ converges to the curvature of $S_\infty$ at $x_0$. This contradicts the assumption that the curvature at the points $x_n$ goes to infinity.
\end{proof}

It follows that the curvature of $S$ is bounded by $-\delta^2<K_S<-\epsilon^2$, where $\delta$ is some constant, whereas we can take $\epsilon=1$. 

\begin{remark}
The main result of this section, Theorem \ref{hyperbolic close to identity}, is indeed a quantitative version of Lemma \ref{curvature bounded below}, which gives a control of how an optimal constant $\Lambda_0$ would vary if $K_0$ is chosen close to $0$.
\end{remark}

\subsection{Schauder estimates} \label{subsection schauder}

By using equation (\ref{hessian hyp}), we will eventually obtain bounds on the principal curvatures of $S$. For this purpose, we need bounds on $u=\sinh d_{\Hyp^3}(\cdot,P_-)$ and its derivatives. Schauder theory plays again an important role: since $u$ satisfies the equation
\begin{equation}
\Delta_S u-2u=0\,. \tag{\ref{lap 2u hyp}}
\end{equation} 
we will use uniform estimates of the form $$||u||_{C^2(B_0(0,\frac{R}{2}))}\leq C ||u||_{C^0(B_0(0,R))}$$ for the function $u$, written in a suitable coordinate system. The main difficulty is basically to show that the operators
$$u\mapsto \Delta_S u-2u$$
are strictly elliptic and have uniformly bounded coefficients.

\begin{prop} \label{schauder estimate hyp}
Let $K_0>1$ and $R>0$ be fixed. There exist a constant $C>0$ (only depending on $K_0$ and $R$) such that for every choice of:
\begin{itemize}
\item A minimal embedded disc $S\subset\Hyp^3$ with $\partial_\infty S$ a $K$-quasicircle, with $K\leq K_0$;
\item A point $x\in S$;
\item A plane $P_-$; 
\end{itemize}
the function $u(\bullets)=d_{\Hyp^3}(\bullets,P_-)$ expressed in terms of normal coordinates of $S$ centered at $x$, namely $$u(z)=\sinh d_{\Hyp^3}(\exp_x(z),P_-)$$
where $\exp_x:\R^2\cong T_{x}S\rar S$ denotes the exponential map, satisfies the Schauder-type inequality
\begin{equation} \label{schauder hyp}
||u||_{C^2(B_0(0,\frac{R}{2}))}\leq C ||u||_{C^0(B_0(0,R))}\,.
\end{equation} 
\end{prop}
\begin{proof}
This will be again an argument by contradiction, using the compactness property. 

 Suppose our assertion is not true, and find a sequence of minimal surfaces $S_n$ with $\partial_\infty(S_n)=\Gamma_n$ a $K$-quasicircle ($K\leq K_0$), a sequence of points $x_n\in S_n$, and a sequence of planes $P_n$, such that the functions $u_n(z)=\sinh d_{\Hyp^3}(\exp_{x_n}(z),P_n)$ have the property that
$$||u_n||_{C^2(B_0(0,\frac{R}{2}))}\geq n ||u||_{C^0(B_0(0,R))}\,.$$
We can compose with isometries $T_n$ of $\Hyp^3$ so that $T_n(x_n)=x_0$ for every $n$ and the tangent plane to $T_n(S_n)$ at $x_0$ is a fixed plane. Let $S_n'=T_n(S_n)$, $\Gamma_n'=T_n(\Gamma_n)$ and $P_n'=T_n(P_n)$. Note that $\Gamma_n'$ are again $K$-quasicircles, for $K\leq K_0$, and the convex hull of each $\Gamma_n'$ contains $x_0$. 

Using this fact, it is then easy to see - as in the proof of Lemma \ref{curvature bounded below} - that one can find $K_0$-quasiconformal maps $\Psi_n$ such that $\Psi_n(S^1)=\Gamma_n'$ and $\Psi_n(1)$, $\Psi_n(i)$ and $\Psi_n(-1)$ are at uniformly positive distance from one another.
Therefore, using Theorem \ref{Compactness property of quasicircles} there exists a subsequence of $\Psi_n$ converging uniformly to a $K_0$-quasiconformal map. This gives a subsequence $\Gamma_{n_k}'$ converging to $\Gamma_\infty'$ in the Hausdorff topology.

By Lemma \ref{compactness harmonic maps}, considering $S_n'$ as images of conformal harmonic embeddings $\sigma_n':\D\rar \Hyp^3$, we find a subsequence of $\sigma_{n_k}'$ converging $C^{\infty}$ on compact subsets to the conformal harmonic embedding of a minimal surface $S_\infty'$. Moreover, by Lemma \ref{curvature bounded below} and Remark \ref{remark compare conformal and geodesic balls}, the convergence is also $C^{\infty}$ on the image under the exponential map of compact subsets containing the origin of $\R^2$.


It follows that the coefficients of the Laplace-Beltrami operators $\Delta_{S_n'}$ on a Euclidean ball $B_0(0,R)$ of the tangent plane at $x_0$, for the coordinates given by the exponential map, converge to the coefficients of $\Delta_{S_\infty'}$. Therefore the operators $\Delta_{S_n'}-2$ are uniformly strictly elliptic with uniformly bounded coefficients. Using these two facts, one can apply Schauder estimates to the functions $u_n$, which are solutions of the equations $\Delta_{S_n'}(u_n)-2u_n=0$. See again \cite{giltrud} for a reference. We deduce that there exists a constant $c$ such that $$||u_n||_{C^2(B_0(0,\frac{R}{2}))}\leq c ||u_n||_{C^0(B_0(0,R))}$$ for all $n$, and this gives a contradiction.
\end{proof}

\subsection{Principal curvatures}

We can now proceed to complete the proof of Theorem \ref{hyperbolic close to identity}. Fix some $w>0$. We know from Section \ref{infinity} that if the Bers norm is smaller than the constant $(1/2)\tanh(w)$, then every point $x$ on $S$ lies on a geodesic segment $l$ orthogonal to two planes $P_-$ and $P_+$ at distance $d_{\Hyp^3}(P_-,P_+)<w$. Obviously the distance is achieved along $l$.

Fix a point $x\in S$. Denote again $u=\sinh d_{\Hyp^3}(\cdot,P_-)$. By Proposition \ref{schauder estimate hyp}, first and second partial derivatives of $u$ in normal coordinates on a geodesic ball $B_S(x,R/2)$ of fixed radius $R/2$ are bounded by 
$C||u||_{C^0(B_S(x,R))}$. 
The last step for the proof is an estimate of the latter quantity in terms of $w$.

We first need a simple lemma which controls the distance of points in two parallel planes, close to the common orthogonal geodesic. Compare Figure \ref{lemmaprojection}.

\begin{lemma} \label{lemma distance formulae}
Let $p\in P_-$, $q\in P_+$ be the endpoints of a geodesic segment $l$ orthogonal to $P_-$ and $P_+$ of length $w$. Let $p'\in P_-$ a point at distance $r$ from $p$ and let $d=d_{\Hyp^3}((\pi|_{P_+})^{-1}(p'),P_-)$. Then
\begin{gather} \label{trig hyp}
\tanh d=\cosh r \tanh w \\
\sinh d=\cosh r \frac{\sinh w}{\sqrt{1-(\sinh r)^2(\sinh w)^2}}\,. \label{expression sin distance width}
\end{gather}
\end{lemma}
\begin{proof}
This is easy hyperbolic trigonometry, which can actually be reduced to a 2-dimensional problem. However, we give a short proof for convenience of the reader. In the hyperboloid model, we can assume $P_-$ is the plane $x_3=0$, $p=(0,0,0,1)$ and the geodesic $l$ is given by $l(t)=(0,0,\sinh t,\cosh t)$. Hence $P_+$ is the plane orthogonal to $l'(w)=(0,0,\cosh w,\sinh w)$ passing through $l(w)=(0,0,\sinh w,\cosh w)$. The point $p'$ has coordinates $$p'=(\cos\theta\sinh r,\sin\theta\sinh r,0,\cosh r)$$ and the geodesic $l_1$ orthogonal to $P_-$ through $p'$ is given by
$$l_1(d)=(\cosh d)(p')+(\sinh d)(0,0,1,0)\,.$$
We have $l_1(d)\in P_+$ if and only if $\langle l_1(d),l'(w)\rangle=0$, which is satisfied for $$\tanh d=\cosh r \tanh w\,,$$ provided $\cosh r \tanh w<1$. The second expression follows straightforwardly.
\end{proof}

\begin{figure}[htbp]
\centering
\includegraphics[height=6.5cm]{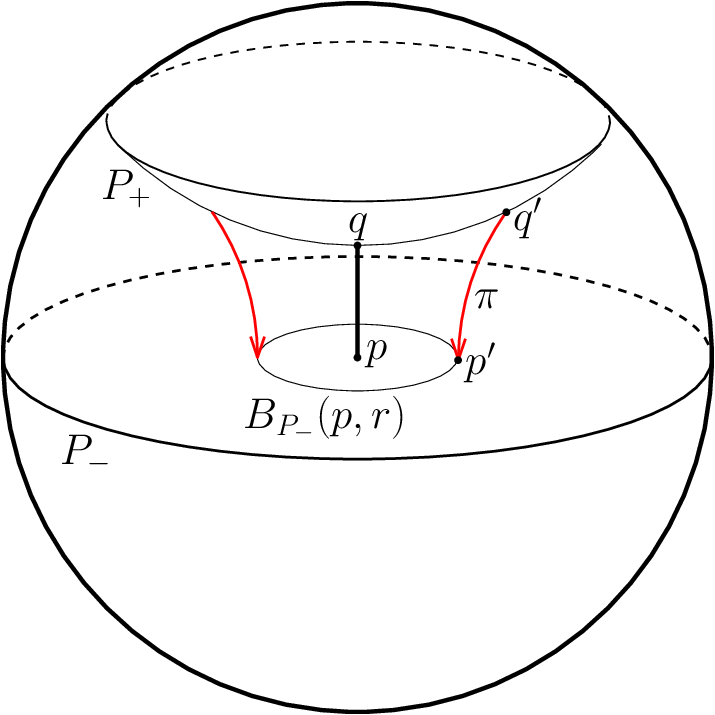}
\caption{The setting of Lemma \ref{lemma distance formulae}. Here $d_{\Hyp^3}(p,p')=r$ and $q'=(\pi|_{P_+})^{-1}(p')$.} \label{lemmaprojection}
\end{figure}

We are finally ready to prove Theorem \ref{hyperbolic close to identity quantitative}. The key point for the proof is that all the quantitative estimates previously obtained in this section are independent on the point $x\in S$.

\begin{reptheorem}{hyperbolic close to identity quantitative}
There exist constants $K_0>1$ and $C>4$ such that
the principal curvatures $\pm\lambda$ of every minimal surface $S$ in $\Hyp^3$ with $\partial_\infty S=\Gamma$ a $K$-quasicircle, with $K\leq K_0$, are bounded by:
\begin{equation} \label{stima finale hyp}
||\lambda||_\infty\leq\frac{C||\Psi||_\mathcal{B}}{\sqrt{1-C||\Psi||_\mathcal{B}^2}}
\end{equation}
\noindent where $\Gamma=\Psi(S^1)$, for $\Psi\in QD(\D)$.
\end{reptheorem}

\begin{proof}
Fix $K_0>1$. Let $S$ a minimal surface with $\partial_\infty S$ a $K$-quasicircle, $K\leq K_0$. Let $x\in S$ an arbitrary point on a minimal surface $S$. By Proposition \ref{computation width bers norm}, we find two planes $P_-$ and $P_+$ whose common orthogonal geodesic passes through $x$, and has length $w=\arctanh(2||\Psi||_{\mathcal{B}})$. 

 Now fix $R>0$.
By Proposition \ref{schauder estimate hyp}, applied to the point $x$ and the plane $P_-$, we obtain that the first and second derivatives of the function 
 $$u=\sinh d_{\Hyp^3}(\exp_x(\bullets),P_-)$$
 on a geodesic ball $B_S(x,R/2)$ for the induced metric on $S$, are bounded by the supremum of $u$ itself, on the geodesic ball $B_S(x,R)$, multiplied by a universal constant $C=C(K_0,R)$.

Let $\pi:\Hyp^3\rar P_-$ the orthogonal projection to the plane $P_-$. The map $\pi$ is contracting distances, by negative curvature in the ambient manifold. Hence $\pi(B_{S}(x,R))$ is contained in $B_{P_-}(\pi(x),R)$. Moreover, since $S$ is contained in the region bounded by $P_-$ and $P_+$, clearly $\sup\{u(x):x\in B_S(0,R)\}$ is less than the hyperbolic sine of the distance of points in $(\pi|_{P_+})^{-1}(B_{P_-}(\pi(x),R))$ from $P_-$. See Figure \ref{projectionhyp}.

\begin{figure}[htbp]
\centering
\includegraphics[height=6.5cm]{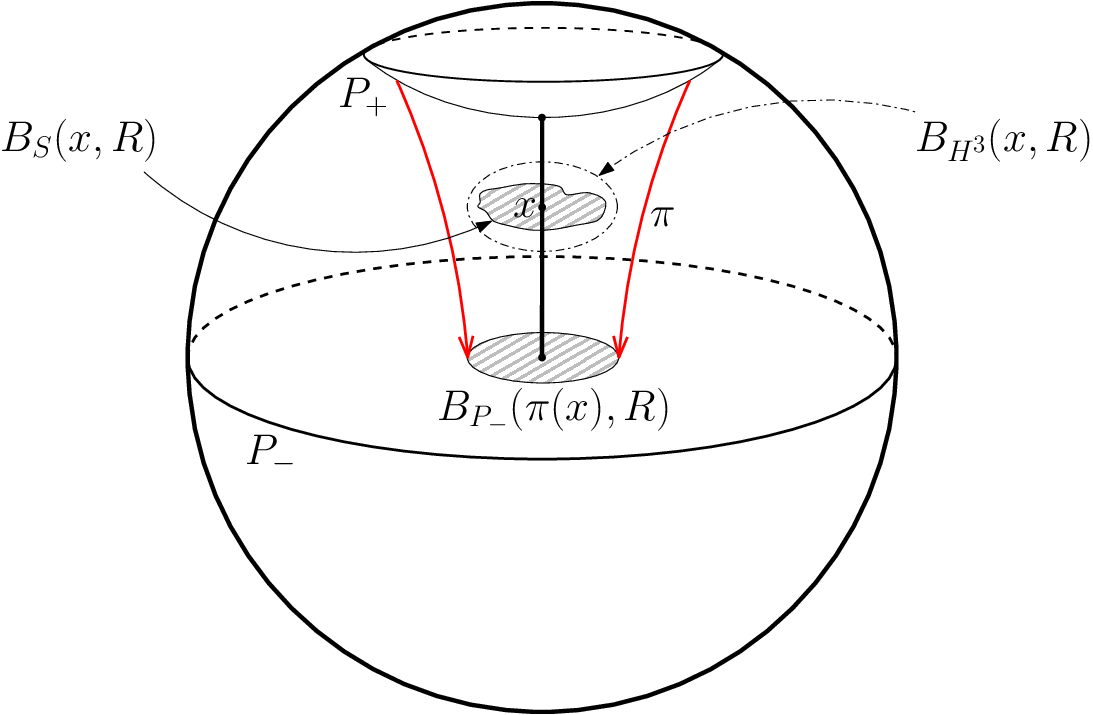}
\caption{Projection to a plane $P_-$ in $\Hyp^3$ is distance contracting. The dash-dotted ball schematically represents a geodesic ball of $\Hyp^3$.} \label{projectionhyp}
\end{figure}

Hence, using Proposition \ref{lemma distance formulae} (in particular Equation \eqref{expression sin distance width}), we get:
\begin{equation} \label{step norm c0 for u final proof}
||u||_{C^0(B_S(x,R))} \leq \cosh R \frac{\sinh w}{\sqrt{1-(\sinh R)^2(\sinh w)^2}}\,,
\end{equation}
where we recall that $w=\arctanh(2||\Psi||_{\mathcal{B}})$.

We finally give estimates on the principal curvatures of $S$, in terms of the complexity of $\partial_\infty(S)=\Psi(S^1)$. We compute such estimate only at the point $x\in S$; by the independence of all the above construction from the choice of $x$, the proof will be concluded. From Equation (\ref{hessian hyp}), we have $$B=\frac{1}{\sqrt{1+u^2-||\grad u||^2}}(\Hess\, u-u\,E)\,.$$
Moreover, in normal coordinates centered at the point $x$, the expression for the Hessian and the norm of the gradient at $x$ are just $$(\Hess u)_i^j=\frac{\partial^2 u}{\partial x^i \partial x^j}\,,\qquad\qquad||\grad u||^2=\left(\frac{\partial u}{\partial x^1}\right)^2+\left(\frac{\partial u}{\partial x^2}\right)^2\,.$$ 
It then turns out that the principal curvatures $\pm\lambda$ of $S$, i.e. the eigenvalues of $B$, are bounded by 
\begin{equation} \label{estimate princ curv bers norm step 1}
|\lambda|\leq \frac{C_1 ||u||_{C^0(B_S(x,R))}}{\sqrt{1-C_1||u||_{C^0(B_S(x,R))}^2}}\,.
\end{equation}
The constant $C_1$ involves the constant  $C$ of Equation \eqref{schauder hyp} in the statement of Proposition \ref{schauder estimate hyp}. 
Substituting Equation \eqref{step norm c0 for u final proof} into Equation \eqref{estimate princ curv bers norm step 1}, with some manipulation one obtains
\begin{equation}
||\lambda||_\infty\leq\frac{C_1(\cosh R)(\tanh w)}{\sqrt{1-(1+C_1)(\cosh R)^2(\tanh w)^2}}\,.
\end{equation}
On the other hand $\tanh w= 2||\Psi||_{\mathcal{B}}$. Upon relabelling $C$ with a larger constant, the inequality
$$||\lambda||_\infty\leq\frac{C||\Psi||_\mathcal{B}}{\sqrt{1-C||\Psi||_\mathcal{B}^2}}$$
is obtained.
\end{proof}

\begin{remark}
Actually, the statement of Theorem \ref{hyperbolic close to identity quantitative} is true for any choice of $K_0>1$ (and the constant $C$ varies accordingly with the choice of $K_0$). However, the estimate in Equation \eqref{stima finale hyp} does not make sense when $||\Psi||^2\geq 1/C$. Indeed, our procedure seems to be quite uneffective when the quasicircle at infinity is ``far'' from being a circle - in the sense of universal Teichm\"uller space. 
Applying Theorem \ref{thm bers emb locally bilipschitz}, this possibility is easily ruled out, by replacing $K_0$ in the statement of Theorem \ref{hyperbolic close to identity quantitative} with a smaller constant.
\end{remark}

Observe that the function $x\mapsto Cx/\sqrt{1-Cx^2}$ is differentiable with derivative $C$ at $x=0$. 
As a consequence of Theorem \ref{thm bers emb locally bilipschitz}, there exists a constant $L$ (with respect to the statement of Theorem \ref{thm bers emb locally bilipschitz} above, $L=1/b_1$) such that
 $||\Psi||_\mathcal{B}\leq Ld_\Teich([\Psi],[\mathrm{id}])$ if $d_\Teich([\Psi],[\mathrm{id}])\leq r$ for some small radius $r$. Then the proof of Theorem \ref{hyperbolic close to identity} follows, replacing the constant $C$ by a larger constant if necessary.
 
\begin{reptheorem}{hyperbolic close to identity}
There exist universal constants $K_0$ and $C$ such that every minimal embedded disc in $\Hyp^3$ with boundary at infinity a $K$-quasicircle $\Gamma\subset\partial_\infty\Hyp^3$, with $K\leq K_0$, has principal curvatures bounded by $$||\lambda||_\infty\leq C\log K\,.$$
\end{reptheorem}

\begin{remark}
With the techniques used in this paper, it seems difficult to give explicit estimates for the best possible value of the constant $C$ of Theorem \ref{hyperbolic close to identity}.
Indeed, in the proof of Theorem \ref{hyperbolic close to identity quantitative}, the constant which occurs in the inequality \eqref{stima finale hyp} depends on the choices of the bound $K_0$ on the maximal dilatation of the quasicircle, and on the choice of a radius $R$. The radius $R$ does not really have a key role in the proof, since the estimate on the principal curvatures is then used only for the point $x$ (in a manner which does not depend on $x$). However, the choice of $R$ is essentially due to the form of Schauder estimates, which provide a constant $C_{\mathrm{Sch}}$ such that 
$$||u||_{C^2(B_0(0,\frac{R}{2}))}\leq C_{\mathrm{Sch}} ||u||_{C^0(B_0(0,R))}\,,$$
where $C_{\mathrm{Sch}}$ depends on the radius $R$. Moreover, $C_{\mathrm{Sch}}$ depends on the bounds on the coefficient of the equation satisfied by $u$, which in our case is 
\begin{equation}
\Delta_S u-2u=0\,. \tag{\ref{lap 2u hyp}}
\end{equation} 
The bound on the coefficients of such equation, which depends on the Laplace-Beltrami operator of the minimal surface $S$, thus depends implicitly on the choice of $K_0$ (a compactness argument was used in this paper, in the proof of Proposition \ref{schauder estimate hyp}).
 Finally, the dependence on the constant $K_0$ appears again in the proof of Theorem \ref{hyperbolic close to identity}, when applying the fact that the Bers embedding is locally bi-Lipschitz (Theorem \ref{thm bers emb locally bilipschitz}). In fact, the local bi-Lipschitz constant depends on the chosen neighborhood of the identity in universal Teichm\"uller space.
\end{remark}

\section{Some applications and open questions} \label{application}

In this section we discuss the proofs of Theorem \ref{theorem uniqueness}, of Corollaries \ref{cor teich dist}, \ref{corollary teichmuller almost fuchsian} and \ref{corollary neighborhood}, and mention some related questions.

\subsection{Uniqueness of minimal discs}
We recall here Theorem \ref{theorem uniqueness}, which was stated in the introduction.

\begin{reptheorem}{theorem uniqueness}
There exists a universal constant $K'_0$ such that every $K$-quasicircle $\Gamma\subset\partial_\infty\Hyp^3$ with $K\leq K'_0$ is the boundary at infinity of a unique minimal embedded disc.
\end{reptheorem}

To prove Theorem  \ref{theorem uniqueness}, one applies the well-known fact that a minimal disc in $\Hyp^3$ with principal curvatures in $[-1+\epsilon,1-\epsilon]$ for some $\epsilon>0$ is the unique one with fixed boundary at infinity. Under this hypothesis, the curve at infinity is necessarily a quasicircle (one can adapt the argument of \cite[Lemma 3.3]{ghw}). For the convenience of the reader, we provide here a sketch of a proof which uses the tools of this paper.

\begin{lemma} \label{lemma uniqueness}
Let $S$ be a minimal embedded disc in $\Hyp^3$ with $\partial_\infty S=\Gamma$. If the principal curvatures of $S$ satisfy $||\lambda||_\infty<1$, then $S$ is the unique minimal disc with $\partial_\infty S=\Gamma$.
\end{lemma}
\begin{proof}[Sketch of proof]
Suppose $\Gamma$ is such that there exists two minimal surfaces $S$ and $S'$ with $\partial_\infty S=\partial_\infty S'=\Gamma$, and that the principal curvatures of $S$ are in $[-1+\epsilon,1-\epsilon]$. As observed after the proof of Lemma \ref{lemma equidistant surfaces hyp}, the  $\rho$-equidistant surfaces from $S$ give a foliation of a convex subset $\mathcal{C}$ of $\Hyp^3$, for $\rho\in (-\arctanh||\lambda||_\infty,\arctanh||\lambda||_\infty)$. By Corollary \ref{minimal surface contained convex hull}, the minimal surface $S'$ is also contained in $\mathcal{C}$.

Now, let $\rho_0$ the supremum of the value of $\rho$ on the minimal surface $S'$. If this supremum is achieved on $S'$, then the minimal surface $S'$ is tangent to the smooth surface $S_{\rho_0}$ at distance $\rho_0$ from $S$. But by Equation \eqref{formulae prin curv equi}, when $\rho>0$ the mean curvature of $S_\rho$ is negative (in our setting, a concave surface, for instance obtained for large positive $\rho$, has negative principal curvatures). Hence by the maximum principle, necessarily $\rho_0\leq 0$. 

If the supremum is not attained, let us pick a sequence of points $x_n\in S'$ such that the value of $\rho$ at $x_n$ converges to $\rho_0$ as $n\to\infty$. One can apply isometries $T_n$ of $\Hyp^3$ so that $x_n$ is mapped to a fixed point $x_0$. By the usual argument (see also Lemma \ref{curvature bounded below}), one can apply Theorem \ref{Compactness property of quasicircles} to ensure that the quasicircles $T_n(\Gamma)$ converge to a quasicircle $\Gamma_\infty$, and then Lemma \ref{compactness harmonic maps} to get the $C^\infty$ convergence on compact sets of the minimal discs $T_n(S')$ to a minimal disc $S'_\infty$ with $\partial_\infty S'_\infty=\Gamma_\infty$, up to a subsequence. Moreover, one can also assume that the minimal discs $T_n(S)$ converge to a minimal disc $S_\infty$. Indeed, consider the points $y_n$ on $S$ such that the geodesic of $\Hyp^3$ through $y_n$, perpendicular to $S$, contains $x_n$. The isometries $T_n$ map $y_n$ to a compact region of $y_n$ (as $d(x_0,T_n(y_n))=d(x_n,y_n)\leq \arctanh||\lambda||_\infty$), thus one can repeat the previous argument (first compose with isometries $R_n$ which map $T_n(y_n)$ to a fixed point $y_0$, and extract a subsequence of $R_n$ converging to an isometry $R_\infty$). By the $C^\infty$ convergence, the minimal surface $S_\infty$ still has principal curvatures in $[-1+\epsilon,1-\epsilon]$, and therefore one can repeat the argument of the previous paragraph, applied to $S_\infty$ and $S'_\infty$, to show that $\rho_0\leq 0$.

In the same way, one proves that the infimum of $\rho$ on $S'$ must be nonnegative, and thus $\rho$ must always be zero on $S'$. This proves that $S=S'$.
\end{proof}

The proof of Theorem \ref{theorem uniqueness} then follows from Lemma \ref{lemma uniqueness}. With respect to the constants $K_0$ and $C$ of Theorem \ref{hyperbolic close to identity}, by choosing some constant $K_0'<\min\{K_0,e^{1/C}\}$  one obtains  that every minimal embedded disc with boundary at infinity a $K$-quasicircle, with $K\leq K_0'$, has principal curvatures bounded by $||\lambda||_\infty<1$.

\subsection{Quasi-Fuchsian manifolds}

In this subsection we collect the applications of Theorem \ref{hyperbolic close to identity} to quasi-Fuchsian manifolds. A quasi-Fuchsian manifold is a Riemannian manifold isometric to $\Hyp^3/G$, where $G$ is subgroup of $\isom(\Hyp^3)$, which acts freely and properly discontinuously on $\Hyp^3$, isomorphic to the fundamental group of a closed surface $\Sigma$, and such that the limit set (i.e. the set of accumulation points in $\partial_\infty\Hyp^3$ of orbits of the action of $G$) is a quasicircle. The topology of a quasi-Fuchsian manifold is $\Sigma\times\R$, where $\Sigma$ is the closed surface. Therefore the results obtained in the previous sections hold for the universal cover $S=\tilde \Sigma_0$ of any closed minimal surface $\Sigma_0$ homotopic to $\Sigma\times\{0\}$.

Recall that Teichmüller space $\Teich(\Sigma)$ of a closed surface $\Sigma$ is the space of Riemann surface structures on $\Sigma$, considered up to biholomorphisms isotopic to the identity. In the same way, the classifying space for quasi-Fuchsian manifolds, which we denote by $\mathcal{QF}(\Sigma)$, is the space of quasi-Fuchsian metrics on $\Sigma\times\R$ up to isometries isotopic to the identity. By the celebrated Bers' Simultaneous Uniformization Theorem (\cite{bers}), $\mathcal{QF}(\Sigma)$ is parameterized by $\Teich(\Sigma)\times \Teich(\Sigma)$. The construction is as follows: since the limit set $\Lambda$ of $G$ is a Jordan curve, the complement of $\Lambda$ in $\partial_\infty\Hyp^3$ has two connected components $\Omega_+$ and $\Omega_-$ on which $G$ acts freely, properly discontinuously and by biholomorphisms. This construction thus provides two Riemann surface structures on $\Sigma$, namely the structures given by the quotients $\Omega_+/G$ and $\Omega_-/G$. Bers proved that these two Riemann surface structures, as points in $\Teich(\Sigma)$, can be prescribed and determine uniquely the quasi-Fuchsian structure in $\mathcal{QF}(\Sigma)$.

Finally, recall that the Teichm\"uller distance between two points of $\Teich(\Sigma)$, namely two Riemann surface structures $\mathcal{A}_1$ and $\mathcal{A}_2$ on $\Sigma$, is defined as:
$$d_{\Teich(\Sigma)}((\Sigma,\mathcal{A}_1),(\Sigma,\mathcal{A}_2))=\frac{1}{2}\inf_{f\sim\mathrm{id}}\log K(f)\,,$$
where $K(f)$ is the maximal dilatation of $f$ and the infimum is taken over all $f:(\Sigma,\mathcal{A}_1)\to(\Sigma,\mathcal{A}_2)$ quasiconformal and isotopic to the identity.

\begin{repcor}{cor teich dist}
There exist universal constants $C>0$ and $d_0>0$ such that, for every quasi-Fuchsian manifold $M=\Hyp^3/G$ with $d_{\Teich(\Sigma)}(\Omega_+/G,\Omega_-/G)<d_0$ and every minimal surface $S$ in $M$ homotopic to $\Sigma\times\{0\}$, the supremum of the principal curvatures of $S$ satisfies:
$$||\lambda||_\infty\leq C d_{\Teich(\Sigma)}(\Omega_+/G,\Omega_-/G)\,.$$
\end{repcor}

Corollary \ref{cor teich dist} follows directly from Theorem \ref{hyperbolic close to identity}. Indeed, let us choose $d_0=(1/2)\log K_0$. If the Teichm\"uller distance between $\Omega_+/G$ and $\Omega_-/G$ is less than $d_0$, then for every $d<d_0$, $d$ larger than the Teichm\"uller distance, one can obtain (by lifting to the universal cover) a $K$-quasiconformal map between $\Omega_+$ and $\Omega_-$ with $K=e^{2d}\leq K_0$. Thus the limit set $\Gamma$ is a $K$-quasicircle, with $K\leq K_0$. Thus by Theorem \ref{hyperbolic close to identity} the lift $S=\tilde\Sigma_0$ of any minimal surface in $M$ satisfies 
$$||\lambda||_\infty\leq C\log K=2Cd$$
Since the choice of $d$ was arbitrary, one obtains
$$||\lambda||_\infty\leq 2C d_{\Teich(\Sigma)}(\Omega_+/G,\Omega_-/G)$$
and the statement is concluded, replacing $C$ by $2C$.

Clearly, the simplest example of quasi-Fuchsian manifolds are Fuchsian manifolds, namely those quasi-Fuchsian manifolds which contain a totally geodesic (and thus minimal) surface homotopic to $\Sigma\times\{0\}$. The lift to $\Hyp^3$ of such surface is a totally geodesic plane, whose boundary at infinity is a circle. Fuchsian manifolds are parameterized by the induced metric on this totally geodesic surface, and thus the space $\mathcal{F}$ of Fuchsian metrics on $\Sigma\times\R$, up to isometry isotopic to the identity, is parameterized by $\Teich(\Sigma)$. As a subset of $\mathcal{QF}$, $\mathcal{F}$ is precisely the diagonal in  $\Teich(\Sigma)\times \Teich(\Sigma)$. 

It is easy to see that the totally geodesic surface in a quasi-Fuchsian manifold is the unique minimal surface. Although the uniqueness of the minimal surface in a quasi-Fuchsian manifold does not hold in general, there is a larger class of manifolds where uniqueness is guaranteed. According to the terminology in \cite{Schlenker-Krasnov}, we have the following definition of almost-Fuchsian manifolds:

\begin{defi}
A quasi-Fuchsian manifold is \emph{almost-Fuchsian} if it contains a minimal surface homotopic to $\Sigma\times\{0\}$ with principal curvatures in $(-1,1)$. 
\end{defi}

We will denote by $\mathcal{AF}(\Sigma)$ the subset of $\mathcal{QF}(\Sigma)$ of almost-Fuchsian manifolds.
Uhlenbeck in \cite{uhlenbeck} first observed that the minimal surface in an almost-Fuchsian manifold is unique. This follows also from the proof of Lemma \ref{lemma uniqueness}, in a simplified version for the compact case. A direct consequence of our results is the following:

\begin{repcor}{corollary teichmuller almost fuchsian}
If the Teichm\"uller distance between the conformal metrics at infinity of a quasi-Fuchsian manifold $M$ is smaller than a universal constant $d_0'$, then $M$ is almost-Fuchsian.
\end{repcor}

Indeed, in Corollary \ref{cor teich dist}, if the Teichm\"uller distance is small enough, then the principal curvatures are bounded by $1$ in absolute value. Finally, if we endow $\mathcal{QF}\cong\Teich(\Sigma)\times \Teich(\Sigma)$ by the 1-product metric, namely
$$d_{\Teich(\Sigma)\times\Teich(\Sigma)}((\mathcal{A}_1,\mathcal{A}'_1),(\mathcal{A}_2,\mathcal{A}'_2))=
d_{\Teich(\Sigma)}(\mathcal{A}_1,\mathcal{A}_2)+d_{\Teich(\Sigma)}(\mathcal{A}'_1,\mathcal{A}'_2)\,,$$
then Corollary \ref{corollary teichmuller almost fuchsian} can be restated by saying that if the distance of a point $(\Omega_+/G,\Omega_-/G)$ from the diagonal is less than $d_0'$, then the quasi-Fuchsian manifold determined by $(\Omega_+/G,\Omega_-/G)$ is almost-Fuchsian. We state this in Corollary \ref{corollary neighborhood} below.

\begin{repcor}{corollary neighborhood}
There exists a uniform neighborhood $N(\mathcal{F}(\Sigma))$ of the Fuchsian locus $\mathcal{F}(\Sigma)$ in $\mathcal{QF}(\Sigma)\cong \Teich(\Sigma)\times \Teich(\Sigma)$ such that $N(\mathcal{F}(\Sigma))\subset \mathcal{AF}(\Sigma)$.
\end{repcor}

\subsection{Further directions}

There is a number of questions left open on quasi-Fuchsian and almost-Fuchsian manifolds. In particular, the results presented in this paper hold for quasi-Fuchsian manifolds such that the two Riemann surfaces at infinity are \emph{close} in Teichm\"uller space. The understanding of the subset of almost-Fuchsian manifolds \emph{far} from the Fuchsian locus is far from being completed. More in general, it is an interesting and challenging problem to understand the geometric behavior of minimal discs in hyperbolic space with boundary at infinity a Jordan curve, especially when this Jordan curve becomes more exotic and phenomena of bifurcations occur.

The techniques of this paper, as observed in Remark \ref{remark width}, motivate towards a definition of \emph{thickness} or \emph{width} of the convex core of a quasi-Fuchsian manifold or, more in general, the convex hull of a quasicircle in $\partial_\infty \Hyp^3$. One might expect to find a relation between such notion of thickness and, for instance, the Teichm\"uller distance between the conformal ends of the quasi-Fuchsian manifold, or the maximal dilatation of the quasicircle. Again, it seems challenging to provide relations which hold far from the Fuchsian locus.



\clearpage
\bibliographystyle{alpha}
\bibliographystyle{ieeetr}
\bibliography{../bs-bibliography}

\end{document}